\documentclass[12pt]{article}
\usepackage{latexsym, amsmath, amssymb, amsthm, amsfonts, amstext, amscd, amsxtra, amsbsy, mathrsfs, stmaryrd, mathdots, bbold}
\usepackage[all]{xy}
\usepackage{tikz}
\usepackage{soul}
\usepackage[makeroom]{cancel}
\usepackage{bbm}
\usepackage{color}
\usepackage{hyperref}
\usepackage{shorttoc}

\usepackage[utf8]{inputenc}
\usepackage[utf8]{inputenc}
\usepackage{amsmath}
\usepackage{amssymb}
\usepackage{amsthm}
\usepackage{youngtab}
\usepackage{tikz}
\usepackage{listings}
\usepackage{amsthm}

\usepackage{tikz-cd}
\usepackage[english]{babel}
 \usepackage{mathtools}
 \usepackage[all]{xy}
\newtheorem{theorem}{Theorem}

\newtheorem{lemma}{Lemma}

\newtheorem{defi}{Definition}
\newtheorem{rmk}{Remark}

\newtheorem{cor}{Corollary}

\date{February 2015}

\title{Loop Grassmannian of quivers and Compactified Coulomb branch of quiver gauge theory with no framing}
\author{Zhijie Dong }
\date{}

\begin{document}

\maketitle
\begin{abstract}
Mirković introduced the notion of loop Grassmannian for symmetric integer matrix $\kappa$. It is a two-step limit of the local projective space $Z_{\kappa}^{\alpha}$, which generalizes the usual Zastava for a simply laced group $G$. The usual loop Grassmannian of $G$ is recovered when the matrix $\kappa$ is the Cartan matrix of $G$.
 On the other hand, Braverman, Finkelberg, and Nakajima showed that the Compactified Coulomb branch $\mathbf{M}_{Q}^{\alpha}$ for the quiver gauge theory with no framing also generalizes the usual Zastava.
We show that in the case when $\kappa$ is the associated matrix of the quiver $Q$, these two generalizations of Zastava coincide, i.e $\mathbf{M}_{Q}^{\alpha}\cong Z_{\kappa(Q)}^{\alpha}$. 


\end{abstract}
\section{Introduction}

\subsection{Generalization of loop Grassmannian}
Let $G$ be a reductive group over a field $k=\overline{k}$ with char($k$)$=0$.
Let $G_{\mathcal{K}}$ and $G_{\mathcal{O}}$ be its formal loop group and formal arc group, respectively. Define $\mathcal{G}(G)=G_{\mathcal{K}}/G_{\mathcal{O}}$ as the loop Grassmannian of $G$, which is an ind-scheme over $k$. 

We consider extensions of the concept of the loop Grassmannian to an arbitrary Kac-Moody group $G_{KM}$. It is not clear how to define the quotient $(G_{KM})_{\mathcal{K}}/(G_{KM})_{\mathcal{O}}$ as an ind-scheme. 
The standard approach is to construct a system of finite dimensional schemes. In \cite{braverman2010pursuing,braverman2019coulomb}, normal slices to certain orbits in the (undefined) loop Grassmannian were considered. In two approaches considered here (\cite{braverman2019coulomb,mirkovic}) these schemes are projective and one constructs the whole loop Grassmannian $\mathcal{G}(G_{KM})$ as a certain colimit \cite{mirkovic}. 

\subsection{Zastava spaces  \cite{finkelberg1997semiinfinite}}

We recall Zastava spaces $Z_G$ for a semisimple simply-connected group $G$. In 1.3 and 1.4 we will consider two constructions of generalization of Zastava for a quiver.
\subsubsection{Intersections of semiinfinite orbits in $\mathcal{G}(G)$}
First, let's get a feeling how the procedure in 1.1 can be done when $G$ is a simply-connected group.

We start by fixing a Cartan subgroup $T$ and a pair of opposite Borel subgroups $(B^{+}, B^{-})$ such that $B^{+}\cap B^{-}=T$. Let $N^{\pm}$ be the unipotent subgroups such that $B^{\pm}=TN^{\pm}.$
For any cocharacter $\alpha$ of $T$, let $t^{\alpha} \in G_{\mathcal{K}}$ be the point corresponding to $\alpha$ and $\overline{t^{\alpha}}$  be the corresponding point in $ \mathcal{G}(G)$. For two cocharacter $\lambda,\mu$ of $T$,
let $S^{+}_{\lambda}$ and $S^{-}_{\mu}$ be the $N_{\mathcal{K}}^{+}$-orbit of the point $\overline{t^{\lambda}}$ and the $N_{\mathcal{K}}^{-}$-orbit of the point $\overline{t^{\mu}}$, respectively. Furthermore, let $\overline{S^{+}_{\lambda}}$ and $\overline{S^{-}_{\mu}}$ be their closures in $\mathcal{G}(G).$

For each cocharacter $\alpha$, we define $F^{\alpha}=\overline{S^{-}_{-\alpha}}\cap \overline{S^{+}_{0}}$. 
 Now for cocharacters $\alpha,\beta$ such that $\alpha\leq \beta$, we have a closed embedding $F^{\alpha}\xhookrightarrow[]{} F^{\beta}$, which we refer to as the growth structure.

Taking the inductive limit of $F^{\alpha}$ for $\alpha\geq 0$, we recover $\overline{S_0}\stackrel{def}{=}\overline{S^{+}_0}$. Let $\overline{S_0}\xrightarrow[]{\times t^{\alpha}}\overline{S_0}$ be the multiplication map that maps $\overline{S_{-\alpha}}\subset \overline{S_0} $ to $\overline{S_0}.$ 
Subsequently, taking the direct limit of $\overline{S_0}\xrightarrow[]{\times t^{\alpha}}\overline{S_0}$ for $\alpha \geq 0$ (the direct system will be referred to as the shift structure), we obtain the entire $\mathcal{G}(G)$ when $G$ is simply connected.

This is a case of the prescription from 1.1 with
 finite-dimensional schemes $F^{\alpha}$ and the ind-system given by the growth structure between $F^{\alpha}$ and the shift structure between $\overline{S_0}.$


\subsubsection{Ordered Zastava $Z'_G$ from global loop Grassmannian}
Consider the global version of the loop Grassmannian $\mathcal{G}(G).$
This is the Beilinson-Drinfeld Grassmannian $\mathcal{G}^{C}(G)$ defined for a reductive group $G$ and a smooth curve $C$. It has a natural projection $\mathcal{G}^{C}(G)\xrightarrow[]{\pi} \sqcup_{n\in \mathbb{N}}C^n$.
This space has the so-called factorization property.
E.g.,
$\mathcal{G}^{C}(G)_{a}\cong \mathcal{G}(G)$ and $\mathcal{G}^{C}(G)_{(a,b)}\cong \mathcal{G}^{C}(G)_{a}\times \mathcal{G}^{C}(G)_{b}$ for $a\neq b.$
For cocharacters $\lambda_1,\cdots, \lambda_n$,
we can also define the global version of $N^{\pm}_{\mathcal{K}}$-orbit\footnote{ First, one defines the section of map $\mathcal{G}^{C}(G)\xrightarrow[]{\pi} C^n$ corresponding to $\lambda_1,\cdots, \lambda_n$\cite[1.1.8]{zhu2009affine}. Then one defines the global loop group of $N$\cite[2.2.3] {achar_riche_central_sheaves}.} $S_{\lambda_1,\lambda_2,\cdots,\lambda_n}^{BD,\pm}\subset \mathcal{G}^{C}(G)$, their orbit closures $\overline{S_{\lambda_1,\lambda_2,\cdots,\lambda_n}^{BD,\pm}}$
and intersections $F_{BD}^{\alpha_1\cdots,\alpha_n}=\overline{S_{-\alpha_1,\cdots,-\alpha_n}^{BD,-}}\cap \overline{S_{0,\cdots,0}^{BD,+}}$ for cocharacters $\alpha_1,\cdots, \alpha_n$. 
E.g.,
the space $S_{\lambda_1,\lambda_2}^{BD}$ is over $C^2$ and its fiber over $(a,b)$ is $S_{\lambda_1}\times S_{\lambda_2}$ for $a\neq b$ and its fiber over $(a,a)$ is $S_{\lambda_1+\lambda_2}$\footnote{The space $\overline{S_{\lambda_1,\lambda_2}^{BD,+}}$ is over $C^2$ and its fiber over $(a,b)$ is $\overline{S_{\lambda_1}}\times \overline{S_{\lambda_2}}$ for $a\neq b$. But here, we only know that the reduced scheme of its fiber over $(a,a)$ is $\overline{S_{\lambda_1+\lambda_2}}$ .}.
For $\alpha=\alpha_{k_1}+\cdots \alpha_{k_n}$, where $\alpha_{k_i}$ is simple coroot,  let $n_i$ be the number of $j's$ such that $\alpha_{k_j}=\alpha_i$. Let $C^{\alpha}=\prod_{i\in I}C^{n_i}$.
Define the ordered Zastava $Z_{G}'^{\alpha}=F_{BD}^{\alpha_{k_1},\cdots,\alpha_{k_n}}$ over $C^{\alpha}$. Here, the space $F_{BD}^{\alpha_{k_1},\cdots,\alpha_{k_n}}$ is canonically isomorphic to $F_{BD}^{\alpha'_{k_1},\cdots,\alpha'_{k_n}}$ for a different order of decomposition $\alpha=\alpha'_{k_1}+\cdots+\alpha'_{k_n}$ (they are different as subspaces of $\mathcal{G}^{C}(G)$, through). Hence the ordered Zastava $Z_{G}'^{\alpha}$ is independent of the order, which justifies the notation.
\subsubsection{Zastava $Z_G$ from quasimaps}\label{Zas}
It turns out that $Z_{G}'^{\alpha}$ and its limit constructions can be defined without using the Beilinson-Drinfeld Grassmannian $\mathcal{G}^{C}(G)$.

First, the affine Zastava\footnote{We will fix a curve $C$ so we drop $C$ from this notation. The affine Zastava was called Zastava in \cite{finkelberg1997semiinfinite} and Zastava was called compactified Zastava in \cite{braverman2019coulomb}.} $Z^{\text{aff},\alpha}_G$ is the space of based quasimaps from the curve $C$ to
the flag variety of $G$ of degree $\alpha\in \mathbb{N}[I]$. It has a natural map $\pi$ to $\mathcal{H}^{\alpha}_{C\times I}$, the Hilbert scheme of point of $C\times I$ of length $\alpha$.
Then Zastava $Z^{\alpha}_G$ is defined as certain compactification (fiberwise with respect to $\pi$) of $Z^{\text{aff},\alpha}_{G}$ ( see remark 3.7 in \cite{braverman2019coulomb} and the reference therein).
Now $Z_{G}^{\alpha}$ is the partially symmetrized version\footnote{The map $Z_{G}'^{\alpha}\xrightarrow[]{\pi}\prod_{i\in I}C^{n_i}$ descends to $Z_{G}^{\alpha}=Z_{G}'^{\alpha}/\prod_{i\in I}S^{n_i}\xrightarrow[]{\pi}\prod_{i\in I}C^{n_i}/S^{n_i}\cong \mathcal{H}^{\alpha}_{C\times I}.$
Define the intersection $F_{\text{aff},BD}^{\alpha_1\cdots,\alpha_n}=S_{\alpha_1,\cdots,\alpha_n}^{BD,-}\cap \overline{S_{0,\cdots,0}^{BD,+}}$. We have $Z_{G}'^{\text{aff},\alpha}=F_{\text{aff},BD}^{\alpha_{k_1},\cdots,\alpha_{k_n}}$.  Then $Z_{G}^{\text{aff},\alpha}$ is the partially symmetrized version of $Z_{G}'^{\text{aff},\alpha}$.}
 of $Z_{G}'^{\alpha}.$
Construction of affine Zastava $Z_{G}^{\text{aff}}$ by quasimap was extended to $G$ being Kac-Moody group in \cite{braverman2006uhlenbeck}, and we expect that the construction of Zastava extends easily.

In this paper, we will consider two other approaches that recover $Z^{\alpha}_{G}$ when $G$ is simply-laced\footnote{It is an open question to identify these two approaches with the quasimap construction for any quiver besides ADE quivers. 
We expect that this identification can be reduced to the identification of three constructions of affine Zastava once the quasimap definition of Zastava has been made.
For identification of the Coulomb branch affine Zastava (see 1.4)
 and quasimap affine Zastava, it suffices to show that the quasimap affine Zastava $\text{Qmap}^{\text{aff},\alpha}_{Q}$ is normal or flat over $\mathcal{H}_{C\times I}$\cite{braverman2019coulomb}.
 This is known when $Q$ is ADE or affine type A \cite{braverman2014semi} where they constructed a resolution of $\text{Qmap}^{\text{aff},\alpha}_{Q}$ and prove it is normal and Cohen-Macaulay, hence flat over $\mathcal{H}_{C\times I}$.}. These two do not directly deal with $G_{KM}$, but with a quiver $Q$.
In the approach in 1.4 below,  the growth and shift structure can also be constructed. Hence, we accomplished the prescription in 1.1 and have a definition of $\mathcal{G}(G_{KM})$ when $G_{KM}$ is associated with a quiver $Q$.

\subsection{Compactified Coulomb branch for quiver gauge theory with no framing}

This is introduced in \cite{braverman2018towards,braverman2019coulomb}.  For any quiver $Q$ (possibly with loops) and dimension vector $\alpha\in \mathbb{Z}[I]$, 
they defined a convolution algebra on certain equivariant homology space of a certain ind-scheme $\mathcal{R}$, and showed it is commutative.
They consider certain positive part $\mathcal{R}^{+}$ of $\mathcal{R}$. 
Denote the Spec of the convolution subalgebra corresponding to $\mathcal{R}^{+}$ by $\textsf{M}^{\alpha}_Q.$
They also defined a filtration on this subalgebra and defined $\mathbf{M}^{\alpha}_Q$ as the Proj of its Rees algebra.
When the underlying diagram of the quiver $Q$ is the Dynkin diagram of $G$, they showed that $
\textsf{M}^{\alpha}_Q
$ ( or $\mathbf{M}^{\alpha}_Q $) coincides with $Z_{G}^{\text{aff},\alpha}$ ( or $Z_{G}^{\alpha}$ resp) for $C=\mathbb{A}^1$ (remark 3.7 in \cite{braverman2019coulomb}).
 However, defining the growth structure in this approach requires some work \cite{muthiah2022fundamental}.
This approach, when accompanied by framing, extends the concept of generalized transversal slices to any quiver $Q$ \cite{braverman2019coulomb}.
\subsection{Local projective spaces \cite{mirkovic2021loop,mirkovic}}
\subsubsection{Motivations} 
We motivate this construction by explaining how to reconstruct $Z^{\alpha}_{G}$ directly from the Cartan matrix of $G$.
We abbreviate $\mathcal{H}_{C\times I}^{\alpha}$ by $\mathcal{H}^{\alpha}$. Now we assume $G$ is simple and simply connected.
 The factorizable part of the Picard group of $\mathcal{G}^{C}(G)$ is $\mathbb{Z}.$
 Let $\mathcal{O}(1)$ be the positive generator of it. Denote the descent of $\mathcal{O}(1)|_{Z'_{G}}$ (under the partial symmetrization map $Z'_{G}\xrightarrow[]{}Z_{G}$) to $Z_{G}$ also by $\mathcal{O}(1).$
 For a sheaf $\mathcal{F}$ on $X$ and a map $X\xrightarrow[]{} Y$, denote by $\mathcal{F}_{X/Y}$ the pushforward of $\mathcal{F}$ from $X$ to $Y$ (the notation is used when the map $X\xrightarrow[]{} Y$ is  clear from the context).

We have the Kodaira embedding $Z^{\alpha}_G\xrightarrow[]{i} \mathbb{P}((\mathcal{O}(1)_{Z^{\alpha}_G/\mathcal{H}^{\alpha}})^{*})$.
Now we further assume that $G$ is simply laced.
A crucial observation is that the restriction map $\mathcal{O}(1)_{Z^{\alpha}_G/\mathcal{H}^{\alpha}}\cong \mathcal{O}(1)_{(Z^{\alpha}_G)^{T}/\mathcal{H}^{\alpha}}$ is an isomorphism. Moreover, $(Z_{G}^{\alpha})^T$ has a moduli description that solely depends on $\alpha$. It is $Gr(\mathcal{T}^{\alpha})$, the Hilbert scheme of the tautological bundle $\mathcal{T}^{\alpha}$ over $\mathcal{H}^{\alpha}$.  The generic fiber of the map $Z^{\alpha}_{G}\xrightarrow[]{\pi} \mathcal{H^{\alpha}}$ is an $n$-th power of $\mathbb{P}^1$'s, where $n$ is the length of $\alpha$. These $\mathbb{P}^1$'s can be regarded as the fibers of $Z^{\alpha_{i_k}}\xrightarrow[]{\pi} \mathcal{H}^{\alpha_{i_k}}$, where $\alpha$ decomposes into the sum of simple roots $\alpha_{i_k}, 1\leq k\leq n, i_k\in I$.

Over the regular part of $C^{\alpha}$\footnote{This is defined as the pullback under $C^{\alpha}\xrightarrow[]{}\mathcal{H}_{C\times I}^{\alpha}$ of the regular part $\mathcal{H}_{reg}^{\alpha}$ of $\mathcal{H}_{C\times I}^{\alpha}$}, the line bundle $\mathcal{O}(1)$ on $Z'^{\alpha}$ is canonically isomorphic to the outer tensor product (over $\mathcal{H}^{\alpha}$) of $\mathcal{O}(1)$ on these $Z^{\alpha_{i_k}},1\leq k\leq n$, regardless of $G$. Denote the isomorphism by $loc$. 
Consider the ordered Zastava $Z'^{\alpha}_G=Z^{\alpha}_G \times_{\mathcal{H}^{\alpha}} C^{\alpha}$ (see footnote 4 on page 3). 
Denote the quotient map $C^{\alpha}\xrightarrow[]{q^{\alpha}}\mathcal{H}^{\alpha}$.
We have

\begin{tikzcd}
   Z_{G,reg}'^{\alpha}\arrow[d,hook] \arrow[r,"\cong"] &\mathbb{P}^{1}\times\cdots \times \mathbb{P}^{1} \arrow[d,hook]  
   \\ \mathbb{P}((\mathcal{O}(1)_{Z'^{\alpha}_{G,reg}/C_{reg}^{\alpha}})^{*}) \arrow[d,"\cong"]  
   & \mathbb{P}((\mathcal{O}(1)\boxtimes \cdots \boxtimes \mathcal{O}(1)_{\mathbb{P}^1\times \cdots \times \mathbb{P}^1/C_{reg}^{\alpha}})^{*})\arrow[d,"\cong"]\arrow[l,"loc","\cong"']
    \\\mathbb{P}((\mathcal{O}(1)_{(q^{\alpha})^{*}Gr(\mathcal{T}_{reg}^{\alpha})/C_{reg}^{\alpha}})^{*}) \arrow[d,hook]
    &\mathbb{P}((\mathcal{O}(1)\boxtimes \cdots \boxtimes \mathcal{O}(1)_{(pr_1)^{*}Gr(\mathcal{T}^{\alpha_{k_1}})\times \cdots \times (pr_n)^{*}Gr(\mathcal{T}^{\alpha_{k_n}})/C^{\alpha}_{reg}})^{*}) \arrow[l,"loc","\cong"']
    \\\mathbb{P}((\mathcal{O}(1)_{(q^{\alpha})^{*}Gr(\mathcal{T}^{\alpha})/C^{\alpha}})^{*}) 
\end{tikzcd}
Here $\mathbb{P}^1$ means $\mathbb{P}^1_{C_{reg}^{\alpha}}$ and $\mathbb{P}^1\times\cdots \times \mathbb{P}^1$ means the fiber product over $C_{reg}^{\alpha}$.
The space $(pr_{i})^{*}Gr(\mathcal{T}^{\alpha_{k_i}})$ is the pullback of $Gr(\mathcal{T}^{\alpha_{k_i}})$ under the projection $C_{reg}^{\alpha}\xrightarrow[]{pr_i} C^{\alpha_{k_i}}=\mathcal{H}^{\alpha_{k_i}}.$
In the left column, the first map is the Kodaira embedding. The second is induced by the isomorphism $\mathcal{O}(1)_{Z'^{\alpha}_G/C^{\alpha}}\cong \mathcal{O}(1)_{(Z'^{\alpha}_G)^{T}/C^{\alpha}}\cong \mathcal{O}(1)_{(q^{\alpha})^{*}Gr(\mathcal{T}^{\alpha})/C^{\alpha}}$.
The third is embedding from the regular part to the whole.
In the right column, the first map is the Kodaira embedding. The second map is induced by $\mathcal{O}(1)_{Z'^{\alpha}_G/C^{\alpha}}\cong \mathcal{O}(1)_{(Z'^{\alpha}_G)^{T}/C^{\alpha}}\cong \mathcal{O}(1)_{(q^{\alpha})^{*}Gr(\mathcal{T}^{\alpha})/C^{\alpha}}$ for $\alpha=\alpha_{k_i}$. 

The map in the second row is denoted by $loc$ which comes from a certain isomorphism (still denoted by) $loc$ between line bundles. Same for the map $loc$ on the next row.
The isomorphism in the first row is induced by the isomorphism $loc$ between line bundles that appear in the second row\footnote{This isomorphism does not descend to an isomorphism $Z_{G,reg}^{\alpha}\cong \mathbb{P}^{1}_{\mathcal{H}_{reg}^{\alpha}}\times_{\mathcal{H}_{reg}^{\alpha}}\cdots \times_{\mathcal{H}_{reg}^{\alpha}} \mathbb{P}^{1}_{\mathcal{H}_{reg}^{\alpha}}$.}

Under the isomorphism $loc$, we get certain sections $S_{can}$ of $\mathcal{O}(1)_{(q^{\alpha})^{*}Gr(\mathcal{T}_{reg}^{\alpha})/C_{reg}^{\alpha}}$ corresponding to the canonical sections of the outer tensor of $\mathcal{O}(1)$ on $\mathbb{P}^1.$

Now it suffices to know how $Z'^{\alpha}_{reg}$ embeds into $\mathbb{P}((\mathcal{O}(1)_{(q^{\alpha})^{*}Gr(\mathcal{T}^{\alpha})/C^{\alpha}})^{*})$, since then we can reconstruct $Z'^{\alpha}$ as the closure of $Z'^{\alpha}_{reg}$ in $\mathbb{P}((\mathcal{O}(1)_{(q^{\alpha})^{*}Gr(\mathcal{T}^{\alpha})/C^{\alpha}})^{*})$ by the fact that the map $Z'^{\alpha}\xrightarrow[]{\pi}C^{\alpha}$ is flat.
This will be determined by the singularities of these sections $S_{can}$ of $\mathcal{O}(1)_{(q^{\alpha})^{*}Gr(\mathcal{T}_{reg}^{\alpha})/C_{reg}^{\alpha}}$ along the diagonal of  $C^{\alpha}$.
The singularities can be classified by a symmetric integer matrix $\kappa$, which is the Cartan matrix of $G$. 
\subsubsection{Construction of Zastava $Z_{\kappa}$ for a matrix $\kappa$}
On the other hand, for any symmetric matrix $\kappa\in M_{I}(\mathbb{Z})$, there is a line bundle $L^{\kappa}$ (\ref{kappa}) on $\mathcal{H}_{C\times I}$ with certain locality property. Let $pr_1,pr_2$ be the maps $\mathcal{H}_{C\times I} \xleftarrow{pr_1} Gr(\mathcal{T}^{\alpha}) \xrightarrow[]{pr_2} \mathcal{H}_{C\times I}$ such that $pr_1(D' \subset D)=D'$ and $pr_2(D'\subset D)=D.$ This defines a vector bundle $\mathcal{I}^{\alpha}=(pr_2)_{*}pr_1^{*}(L^{\kappa})$.
Denote the dual of $\mathcal{I}^{\alpha}$ by $V^{\alpha}$. 
Its locality structure will determine
an embedding $\Pi^{n}_{i=1} \mathbb{P}_{\mathcal{H}_{reg}^{\alpha}}^1\xrightarrow[]{i} \mathbb{P}(V^{\alpha})$.
Then one defines the space $Z^{\alpha}_{\kappa}$ as the closure of the image of $i$ in $\mathbb{P}(V^{\alpha})$.

As we briefly discussed above,
when the matrix $\kappa$ is the Cartan matrix of $G$, Mirkovic showed that $
Z^{\alpha}_{\kappa}$
 coincides with $Z_{G}^{\alpha}$\footnote{ This is done by showing $(q^{\alpha})^{*}Z^{\alpha}_{\kappa} \cong (q^{\alpha})^{*}Z_{G}^{\alpha}=Z_{G}'^{\alpha}$. It seems natural to use the notation $Z_{\kappa}'^{\alpha}$ for $(q^{\alpha})^{*}Z^{\alpha}_{\kappa}$ as we do for $Z_{G}'^{\alpha}$ but we will not use that.}.
In principle, the "$\kappa$-approach" is more general than the Coulomb branch construction since we can consider any symmetric matrix that is not necessarily the (modified) incidence matrix of a quiver (possibly with loops). In the upcoming paper\cite{dong}, we will provide some examples where the projection $
Z^{\alpha}_{\kappa}\xrightarrow[]{\pi}\mathcal{H}^{\alpha}_{C\times I}$ is not flat when the matrix $\kappa$ is not the incidence matrix of a quiver. Therefore, the merit of zastava for such $\kappa$ is subject to skepticism.

The drawback of this approach is that $Z_{\kappa}^{\alpha}$ is defined as a closure and we lack a direct description. Additionally, it is unclear how to define the generalized transverse slices of $\mathcal{G}(G_{KM})$. However, the advantage is that the growth structure \cite{mirkovic} is much easier to define.

\subsection{Identification of two generalizations of Zastava}
Now Let $C=\mathbb{A}^1$.
The main result of this paper is that the local (projective) space construction coincides with the Coulomb branch construction when $\kappa$ is the incidence matrix $\kappa(Q)$ of a quiver $Q$.
We will clarify the following parallel structures for Coulomb branch $\mathbf{M}^{\alpha}_{Q}$ 
and local projective space $Z^{\alpha}_{\kappa(Q)}$. In the table below, we omit the subscripts $Q$ and $\kappa(Q)$. Also, we fix $\alpha=(n_i)_{i\in I}$. Let $G=\prod_{i\in I}GL(k^{n_i})$ and $T$ be the diagonal subgroup of $G$ \footnote{These $G, T$ depend on $\alpha$ and are different from the $G, T$ that appeared in the Zastava before. It is clear which $G,T$ we mean from the context.}. Here, we omit $\alpha$ for the notations involving $\mathcal{R}$.

\begin{center}
\begin{tabular}{ |c|c|c| } 
 \hline
        & $\mathbf{M}^{\alpha}$ & $Z^{\alpha}$ \\ 
         \hline
 projective embedding & $\mathbf{M}^{\alpha}\xhookrightarrow[]{j_\alpha} 
 \mathbb{P}_{\mathcal{H}^{\alpha}}(H_{*}^{{G}}(\mathcal{R}^{+}_{\leq 1})^{*})$ &$Z^{\alpha}\xhookrightarrow[]{i_\alpha}  \mathbb{P}_{\mathcal{H}^{\alpha}}(V^{\alpha})$ \\ 
  \hline
regular fibers over $\mathcal{H}_{reg}$ & product of $\mathbb{P}^1$'s &  product of $\mathbb{P}^1$'s \\ 
 \hline
global basis   & basis of $H_{*}^{G}(\mathcal{R}^{+}_{\leq 1})$ &  basis of $\Gamma(\mathcal{H}^{\alpha},\mathcal{I}^{\alpha})$ \\ 
   \hline
\end{tabular}
\end{center}
Here $\mathcal{R}^{+}_{\leq 1}$ is certain subspace of the positive part $\mathcal{R}^{+}$ of $\mathcal{R}.$
The embedding $j_{\alpha}$ follows from Lemma \ref{generate}, while the embedding $i_{\alpha}$ is based on the definition of the local space.
The space $\mathbf{M}^{\alpha}$ is over $Spec(H_{G}^{*}(pt))\cong \mathcal{H}^{\alpha}.$ Let $C^{\alpha}=\prod_{i\in I}C^{n_i}.$ 

Let us see how to identify $\mathbf{M}^{\alpha}$ and $Z^{\alpha}$ after pulling back under $C^{\alpha}\xrightarrow[]{q^{\alpha}} \mathcal{H}^{\alpha}$.

Under the pull back $q^{\alpha}$, the space
$H_{*}^{G}(\mathcal{R}^{+}_{\leq 1})$ becomes $H_{*}^{T}(\mathcal{R}^{+}_{\leq 1})$ and $\Gamma(\mathcal{H}^{\alpha},\mathcal{I}^{\alpha})$ becomes
$\Gamma(C^{\alpha},(q^{\alpha})^{*}\mathcal{I}^{\alpha}).$
We will identify $H_{*}^{T}(\mathcal{R}^{+}_{\leq 1})\xrightarrow[\cong]{rel}\Gamma(C^{\alpha},(q^{\alpha})^{*}\mathcal{I}^{\alpha})$ in section \ref{identify projective space},  as free rank 1 modules\footnote{In each component.} over $H^{*}_{T}(\mathcal{R}^{+}_{\leq 1})
\cong 
\mathcal{O}((q^{\alpha})^{*}Gr(\mathcal{T}^{\alpha})).$

In component, the isomorphism $H_{*}^{T}(\mathcal{R}_{\varpi_{\beta}})\xrightarrow[\cong]{rel}\Gamma(C^{\alpha},(q^{\alpha})^{*}\mathcal{I}^{\alpha,\beta})$ is given by the ring isomorphism $H_{*}^{T}(\mathcal{R}_{\varpi_{\beta}})\cong \mathcal{O}((q^{\alpha})^{*}Gr^{\beta}(\mathcal{T}^{\alpha}))$ and choosing basis $[\mathcal{R}_{\varpi_{\beta}}]$ of $H_{*}^{T}(\mathcal{R}_{\varpi_{\beta}})$ and basis $u^{\beta}\otimes 1$ of $\Gamma(C^{\alpha},(q^{\alpha})^{*}\mathcal{I}^{\alpha,\beta})$ (see \ref{identify projective space} for the notations).

By a simple lemma (Lemma \ref{lemma}), it suffices to prove that the dotted arrow below exists as an isomorphism, i.e., the isomorphism $rel$ restricts to an isomorphism of two embeddings over the regular part $\mathcal{H}_{reg}$(lemma \ref{local iso}).
\begin{center}
\begin{tikzcd}
   (q^{\alpha})^{*}Z_{reg}^{\alpha}\arrow[d,dashed]{rel_{reg}}\arrow[r,hook] & (q^{\alpha})^{*}Z^{\alpha}\arrow[r,hook,"i_{\alpha}"] & \mathbb{P}_{C^{\alpha}}((q^{\alpha})^{*}V^{\alpha})\arrow[d,"rel\simeq"]{rel}\\
    (q^{\alpha})^{*}\mathbf{M}_{reg}^{\alpha} \arrow[r,hook] & (q^{\alpha})^{*}\mathbf{M}^{\alpha} \arrow[r,hook,"j_\alpha"]  & \mathbb{P}_{C^{\alpha}}(H_{*}^{T}(\mathcal{R}^{+}_{\leq 1})^{*}) 
\end{tikzcd}
\end{center}
To produce the dotted arrow, we will show their homogeneous coordinate rings are isomorphic.
Any basis of
$H_{*}^{T}(\mathcal{R}^{+}_{\leq 1})$
(resp.
$\Gamma(C^{\alpha},(q^{\alpha})^{*}\mathcal{I}^{\alpha})$)
over  $H^{*}_{T}(\mathcal{R}^{+}_{\leq 1})$
(resp.
$\mathcal{O}((q^{\alpha})^{*}Gr(\mathcal{T}^{\alpha}))$
generate the homogeneous coordinate rings of $\mathbb{P}_{C^{\alpha}}(H_{*}^{T}(\mathcal{R}^{+}_{\leq 1})^{*})$ (resp. $\mathbb{P}_{C^{\alpha}}((q^{\alpha})^{*}\mathcal{I}^{\alpha})$) over $\mathcal{O}(C^{\alpha}).$
There are canonical choices of basis for the  localization over $C_{reg}$ of both $H_{*}^{T}(\mathcal{R}^{+}_{\leq 1})$ and 
$\Gamma(C^{\alpha},(q^{\alpha})^{*}\mathcal{I}^{\alpha})$, which we will call the local bases\footnote{There are no canonical choices of basis before localization. We call any of them a global basis. Here we use parallel terminologies for global and local basis but they have different natures}. 
The local basis of $H_{*}^{T}(\mathcal{R}^{+}_{\leq 1})$ arises naturally from the localization theory of equivariant homology, and the local basis of $\Gamma(C^{\alpha},(q^{\alpha})^{*}\mathcal{I}^{\alpha})$ is inherent in the definition of local projective spaces.
We then show that these are identified via $rel$.
Finally, we verify that they satisfy the same relations through direct computation, using our explicit understanding of multiplication formulas on both sides.


\subsection{Background and further directions}
The loop Grassmannians associated to quivers $Q$ contain a class of loop Grassmannians for affine groups, so it is a part of a current effort by mathematicians and physicists to lift features of Langlands program to dimension 2.
This paper unifies approaches to $\mathcal{G}_{Q}$ via Coulomb branches and local spaces. For any quiver $Q$ the latter 
 constructs Zastavas and Beilinson-Drinfeld Grassmannian \cite{mirkovic} (more generally, for symmetric integral matrices), and the former \cite{braverman2019coulomb} constructs Zastava and generalized slices.

This seems to make some features of Satake equivalence for general quiver $Q$ more accessible at the moment.
This includes conjectural geometric constructions of the positive part $U(\mathfrak{g}^{+}_{Q})$ of the enveloping algebra for a Kac-Moody algebra associated with $Q$ (using Zastava) \cite[5.4]{finkelberg2020drinfeld}, and of its irreducible highest weight module $V(\lambda)$ (using generalized slices) in \cite[Conjecture 3.25]{braverman2019coulomb}. Also, one hopes that the Geometric Casselman-Shalika theorem, i.e., the Whittaker version of Satake equivalence \cite[6.36.1]{raskin2021chiral} can be extended to loop Grassmannian $\mathcal{G}_{Q}$ associated to quivers.

\section{Mirković local projective space (Zastava)}
We recall the notion of local space and local projective space introduced by Mirkovic\cite{mirkovic2021loop,mirkovic}.
\subsection{Hilbert scheme $\mathcal{H}_{C\times I}$}
Let $I$ be a finite set. 
Let $C$ be a smooth algebraic curve over an algebraic closed field $k$. We consider the Hilbert scheme $\mathcal{H}_{C\times I}$, the moduli of finite subschemes of $C\times I$. 
Let $C_i=C\times i\subset C\times I$ and $\mathcal{H}^{n_i}_{C_i}$ be the Hilbert scheme of finite subschemes $D$ of length $n_i$ of $C_i$.
For $\alpha=(n_i)_{i\in I}\in \mathbb{N}[I],$ we define $\mathcal{H}^{\alpha}_{C\times I}=\prod_{i\in I} \mathcal{H}^{n_i}_{C_i}.$ We understand $\alpha$ as a dimension vector.

We have $\mathcal{H}_{C\times I}=\sqcup_{\alpha\in \mathbb{N}[I]} \mathcal{H}^{\alpha}_{C\times I}.$
We also call $D=(D_i)_{i\in I}\in \mathcal{H}_{C\times I}$ an $I$-colored divisor where $D_i\in H_{C_i}$ is an effective divisor of $C_i$ for $i\in I$.
Since $C$ is a smooth algebraic curve, the Hilbert scheme $\mathcal{H}^n_{C}$ can also be viewed as the $n$-th symmetric power $S^{n}C$ of $C$, which is also the categorical quotient $C^n//S_n.$

\subsection{The regular part of $\mathcal{H}_{C\times I}$}
For $i,j\in I$ we define the $(i,j)$-diagonal divisor 
$\Delta_{ij}\subset \mathcal{H}_{C\times I}$. When $i\neq j$ the condition for $D\in \Delta_{ij}$ is that the component divisors $D_i$ and $D_j$ meet.
Also, $\Delta_{ii}$ is a discriminant divisor, a subscheme $D\in  \mathcal{H}_{C\times I}$ lies in $\Delta_{ii}$ if $D_i$ is not discrete, i.e., some point has multiplicity $>1$.
We call the complement of the union $\cup\Delta_{ij}\subset \mathcal{H}_{C\times I}$ the regular part of $\mathcal{H}_{C\times I}$ and denote it by $\mathcal{H}_{reg}$.
We call a divisor $D\in \mathcal{H}_{reg}$ a regular divisor. We say two divisors $D$ and $D'$ are disjoint if they are disjoint after forgetting the colors.
\\

\subsection{Local line bundle $L$ on  $\mathcal{H}_{C\times I}$.}

\begin{defi}
 A locality structure of a vector bundle $\mathcal{V}$ over $\mathcal{H}_{C\times I}$ is a system of isomorphisms : For two disjoint divisors\footnote{Here, divisor means closed subschemes of $C\times I\times S$ that are flat over $S$ } $D,D'\in \mathcal{H}_{C\times I}$,
$$\mathcal{V}_{D}\otimes \mathcal{V}_{D'} \xrightarrow[]{i_{D,D'}}\mathcal{V}_{D\cup D'} $$ that satisfy the associative, commutative and unital properties.\\ 
Similarly, a local structure on a space $Y$ over $\mathcal{H}_{C\times I}$ is a system of isomorphisms
$$Y_{D}\times Y_{D'} \xrightarrow[]{i_{D,D'}} Y_{D\cup D'}$$ that satisfy the same properties.
\end{defi}

\begin{rmk}\label{locality1}
The above structure (map between functors) is represented by the following.
For any $\alpha=\alpha_1+\alpha_2$
, denote the vector bundle $\mathcal{V}$ on the component $\mathcal{H}^{\alpha}$ by $\mathcal{V}^{\alpha}.$\footnote{Here in the notation $\mathcal{V}^{\alpha}$, the superscript $\alpha$ indexes the connected component, not power.}
We have the addition of divisors
$$\mathcal{H}^{\alpha_1}\times \mathcal{H}^{\alpha_2} \xrightarrow[]{a^{\alpha_1,\alpha_2}} \mathcal{H}^{\alpha}.$$
Let $pr_i, i=1,2$ and   
$\mathcal{H}^{\alpha_1}\times \mathcal{H}^{\alpha_2} \xrightarrow[]{pr_i} \mathcal{H}^{\alpha_i}$ be the $i$-th projection.
Denote $(\mathcal{H}^{\alpha_1}\times \mathcal{H}^{\alpha_2})_{disj}$ by the open part of $\mathcal{H}^{\alpha_1}\times \mathcal{H}^{\alpha_2}$ that consists $(D,D')$ such that $D\in \mathcal{H}^{\alpha_1},D'\in \mathcal{H}^{\alpha_2}$ are disjoint.
A locality structure on a vector bundle $\mathcal{V}$ over $\mathcal{H}_{C\times I}$ is a system of isomorphisms over $(\mathcal{H}^{\alpha_1}\times \mathcal{H}^{\alpha_2})_{disj}$
$$\mathcal{V}^{\alpha_1}\boxtimes \mathcal{V}^{\alpha_2}:=pr_1^{*}\mathcal{V}^{\alpha_1}\otimes pr_{2}^{*}\mathcal{V}^{\alpha_2}
\xrightarrow[]{i^{\alpha_1,\alpha_2}} (a^{\alpha_1,\alpha_2})^{*}\mathcal{V}^{\alpha},$$
 for any $\alpha_1,\alpha_2$
that satisfy the associative, commutative and unital properties.
Later, we will have vector bundles with the notation $L,\mathcal{I}$ and $V$, where $L$ means it is of rank 1, i.e. a line bundle. 
\end{rmk}

\subsection{Induced vector bundle $\mathcal{I}$}
 To a local line bundle $L$ over $\mathcal{H}_{C\times I}$,
 we associate an induced vector bundle $\mathcal{I}$ over $\mathcal{H}_{C\times I}$.
 
Let $\mathcal{T}\subset \mathcal{H}_{C\times I}\times (C\times I)$ be the tautological scheme over $\mathcal{H}_{C\times I}$, where the fiber $\mathcal{T}_{D}$ at $D\in \mathcal{H}_{C\times I}$ is $D\subset C\times I.$
 Let $Gr(\mathcal{T})$ be
 the relative Hilbert scheme $Gr(\mathcal{T}/\mathcal{H}_{C\times I})$
 over $\mathcal{H}_{C\times I}$ such that the fiber at $D\in \mathcal{H}_{C\times I}$ is $Gr(D):=\mathcal{H}_{D}$, where $\mathcal{H}_D$ is the Hilbert scheme of all subschemes $D'\subset D$.
 
Let $\mathcal{T}^{\alpha}$ be the component of $\mathcal{T}$ over $\mathcal{H}^{\alpha}_{C\times I}$  and $Gr^{\beta}(\mathcal{T}^{\alpha})$ be the component of $Gr(\mathcal{T}/\mathcal{H}_{C\times I})$ whose fiber at $D$ is $Gr^{\beta}(D):=\mathcal{H}^{\beta}(D),$ the Hilbert scheme of subschemes $D'\subset D$ of length $\beta$.

The scheme $Gr(\mathcal{T})$ is a self correspondence of $\mathcal{H}_{C\times I}$,
 $$\mathcal{H}_{C\times I} \xleftarrow{pr_1} Gr(\mathcal{T}) \xrightarrow[]{pr_2} \mathcal{H}_{C\times I},$$
where $pr_1(D' \subset D)=D'$ and $pr_2(D'\subset D)=D.$
\begin{lemma}
 The map $pr_2$ is finite flat so $\mathcal{I}$ is a vector bundle.  
\end{lemma}
\begin{proof}
For brevity, assume $I=\{1\}$. 
For $p=(p_1<\cdots<p_k)$, let $Gr_p(\mathcal{T})$ be the partial flag space of all filtrations $D_1\subset \cdots \subset D_k\subset D$ with $|D_i|=p_i.$ The maps $Gr_q(\mathcal{T})\xrightarrow[]{} Gr_p(\mathcal{T})$ are clearly finite flat when $q$ is obtained from $p$ by adding $p_{s+1}$ after some $p_s$. 
Build a tower of maps from $p=\{0\}$ to $p=(1,\cdots,n)$. The claim follows by the property that if $X\xrightarrow[]{f} Y$ is flat, $Y\xrightarrow[]{g} Z$ is flat if and only if $x\xrightarrow[]{g\circ f} Z$ is flat.
\end{proof}
Now define the induced vector bundle $\mathcal{I}=(pr_2)_{*}pr_1^{*}(L)$ and its dual $V=\mathcal{I}^{*}$.
\begin{lemma}\label{pre local of I}
    A local structure of $L$ canonically induces a local structure on $V$
\end{lemma}
\begin{proof}
The locality structure is naturally induced from its dual vector bundle so we check the locality structure for $\mathcal{I}$.
Here we only check the case where $D,D'$ are regular for later use. The general case is similar.
We compute the fiber $\mathcal{I}_D$ for $D=\sqcup^{n}_{k=1} p_k$ 
$$\mathcal{I}_D=((pr_2)_{*}(pr_1)^{*}L)_D
=\oplus_{E\subset D} L_{E}.$$
Now for disjoint divisors $D,D'$, 
$$\mathcal{I}_{D}\otimes \mathcal{I}_{D'}\cong(\oplus_{E\subset D} L_{E})\otimes (\oplus_{E'\subset D'} L_{E'})\cong \oplus_{E\subset D,E'\subset D'} L_{E}\otimes L_{E'}\xrightarrow[]{i^{\mathcal{I}}_{D,D'}} \oplus_{E\cup E'\in D\cup D'} L_{E\cup E'}\cong \mathcal{I}_{D\cup D'},$$
where $i^{\mathcal{I}}_{D,D'}$ is the direct sum of the locality isomorphisms of $L$.
We have the last isomorphism since $D, D'$ are disjoint, any subscheme $F\subset D\cup D'$ can be uniquely written as the union of $E\subset D$ and $E'\subset D'$.
\end{proof}

\subsection{Local projective space of a local vector bundle }\label{def of zastava}
To a local vector bundle $V$ over $\mathcal{H}_{C\times I}$ we will associated its local projective space $\mathbb{P}^{loc}(V)\subset \mathbb{P}(V)$.\\

First, at a point $D=p\in C\times I\subset \mathcal{H}_{C\times I}$  the locality condition is empty, so we define $\mathbb{P}^{loc}(V)_{p}=\mathbb{P}(V_{p})$. Now over a regular divisor $D=\sum p_k$, as required by locality condition we define
$$\mathbb{P}^{loc}(V)_{D}=\prod^{n}_{k=1}\mathbb{P}(V_{p_k}).$$
Then the locality structure of $\otimes^{n}_{k=1} V_{p_k}\cong V_D$ gives the Segre embedding 
$$\mathbb{P}^{loc}(V)_{D}\hookrightarrow \mathbb{P}(V)_{D}.$$
Now we have a subspace ${P}^{loc}_{reg}(V)$ in $\mathbb{P}(V)$ defined over $\mathcal{H}_{reg}.$
We define the local projective space $\mathbb{P}^{loc}(V)$ as the closure of $\mathbb{P}^{loc}_{reg}(V)$ in $\mathbb{P}(V).$
By definition $\mathbb{P}^{loc}_{reg}(V)$ is a local space over $\mathcal{H}.$
\begin{rmk}\label{prod2}
Now we represent the Segre embedding.
Denote by $e_i \in \mathbb{N}[I]$ which is $1$ in the $i$-th coordinate and $0$ in the others. 
Denote $C^{\alpha}=\prod_{i\in I}C_i^{n_i}$
Let $q^{\alpha}$ be the quotient map $C^{\alpha}\xrightarrow[]{q^{\alpha}}\mathcal{H}^{\alpha}$.
For $i\in I, 1\leq j \leq n_i$,
let $ pr_{ij}$ be the compositions of  \( C^{\alpha} \xrightarrow[]{\text{i-th projection}} C^{n_i}_{i} \xrightarrow[]{\text{j-th projection}} C_i \).

Let $C^{\alpha}_{reg}=(q^{\alpha})^{-1}(\mathcal{H}^{\alpha}_{reg})$ and $C_{reg}=\sqcup_{\alpha\in \mathbb{N}[I]} C^{\alpha}_{reg}$.
Recall we denote the restriction of $V,\mathcal{I}$ to $C_i$ by $V^{e_i},\mathcal{I}^{e_i}$.
For $\alpha=(n_i)_{i\in I}$, the local structure of $\mathcal{I}$ defines the isomorphism over $C_{reg}$
\begin{equation}\label{local I}
    \bigotimes_{i\in I} \bigotimes_{j\in \{1,\cdots a_i\}}pr_{ij}^{*}\mathcal{I}^{e_i} \xleftarrow[]{i_{\alpha}}(q^{\alpha})^{*}\mathcal{I}^{\alpha}.
    \end{equation}
and the  
Segre embedding over $C_{reg}$ is the corresponding embedding
\begin{equation}\label{embedding V}
\prod_{i\in I} \prod_{j\in \{1,\cdots a_i\}}pr_{ij}^{*}\mathbb{P}(V^{e_i}) \xhookrightarrow{\mathbf{i}^{*}_{\alpha}} (q^{\alpha})^{*}\mathbb{P}(V^{\alpha}).
 \end{equation}
It is more convenient to work over $C^{\alpha}$ other than $\mathcal{H}^{\alpha}$ when doing calculations. We will always pullback objects over $\mathcal{H}^{\alpha}$ by $q^{\alpha}$. 
In practice, we will consider the image of $\mathbf{i}^{*}_{\alpha}$ and its closure over $C^{\alpha}$ which descents to $\mathbb{P}^{loc}_{reg}(V)$ and $\mathbb{P}^{loc}(V)$ respectively under $q^{\alpha}$.
\end{rmk}
 \begin{rmk}\label{fiber dim}
We compute the fiber $V_{p}$ for $p$ a point.
Since $Gr(p)=\{\emptyset,p\}$, we have $V_{p}=L^*_{\emptyset}\oplus L^*_{p}\cong k\oplus L^{*}_{p}$ and $\mathbb{P}(V_{p})\cong \mathbb{P}^1.$
By locality, the fiber $V_D$ at a regular divisor $D$ of length $\alpha$ is\footnote{Since we just compute the fiber we can ignore the index $I$ and only count the total length $|\alpha|$ of $D$.} 
$$V_D=(\oplus_{D'\subset D} L_{D'})^{*}\cong k^{2^{|\alpha|}}.$$
Hence over regular divisors, the fiber of $\mathbb{P}^{loc}(V^{\alpha})$ is a product of $\mathbb{P}^1$'s and the local space $\mathbb{P}^{loc}(V^{\alpha})$
can be viewed as a degeneration of a product of $\mathbb{P}^1$'s in $\mathbb{P}^{2^{|\alpha|}-1}.$
 \end{rmk}
 
\subsection{Zastava space $Z_{\kappa}$ of a symmetric matrix $\kappa$}\label{kappa}
We define a local line bundle over $\mathcal{H}_{C\times I}$ associated with a symmetric integral matrix $\kappa\in M_{I}(\mathbb{Z})$ as
$$L_{\kappa}=\mathcal{O}_{\mathcal{H}}(-\kappa\Delta),$$
where $\kappa\Delta$ is the divisor $\sum_{i\leq j}\kappa_{ij}\Delta_{ij}$ in $\mathcal{H}_{C\times I}.$
On each connected component, we can embed $L_{\kappa}$ into the sheaf of the fraction field of $C$.
The locality structure of $L$ is then the multiplication of rational functions.

For the local line bundle $L_{\kappa}$, denote the induced vector bundle by $\mathcal{I}_{\kappa}$ and its dual by $V_{\kappa}$.
We define $Z^{\alpha}_{\kappa}=\mathbb{P}^{loc}(V_{\kappa}^{\alpha}). $ Now we fix a matrix $\kappa$ and drop $\kappa$ from the notation.
Denote the projection $Z^{\alpha}\xrightarrow[]{\pi}\mathcal{H}^{\alpha}.$ Abusing notation, we will also denote the pullback of $\pi$ under $q^{\alpha}$ by $\pi$. 
\subsection{ $Z^{\alpha}$ for $C=\mathbb{A}^1$ }\label{compute}
From now on, we set $C=\mathbb{A}^1$ to be the affine line.
We fix a dimension vector $\alpha$. Since the matrix $\kappa$ is also fixed we write $L=L_{\kappa},\mathcal{I}=\mathcal{I}_{\kappa},V=V_{\kappa}.$ 
Denote by $\mathcal{S}_{\mathcal{H}^{\alpha}} (\mathcal{I}^{\alpha})$ the symmetric algebra generated by $\mathcal{I}^{\alpha}$ over $\mathcal{H}^{\alpha}$. Let $A_0$ be\footnote{The notation $A_0$ is only used in section \ref{compute}. It does not reflect its dependence on $\alpha$. Since we fix $\alpha$ in \ref{compute}, this will not cause confusion. The same happens for notation $\mathbf{I}$.} its pullback under $q^{\alpha},$ hence $Proj(A_0)\cong (q^{\alpha})^{*}\mathbb{P}(V^{\alpha}).$
Let $A_{0,reg}$ be its localization over $C^{\alpha}_{reg}=(q^{\alpha})^{*}\mathcal{H}_{reg}$.\\
Let $\mathbf{I}_{reg}$ be an ideal in $A_{0,reg}$ such that $Proj(A_{0,reg}/\mathbf{I}_{reg})\cong (q^{\alpha})^{*}Z^{\alpha}|_{reg}.$\\ 
Since $\mathcal{I}^{\alpha}$ is a free sheaf over $\mathcal{H}^{\alpha}$, once we choose a trivialization of it, we can write the algebra $A_0$ as a polynomial ring.


\subsubsection{Explicit description of locality structure on $\mathcal{I}$}
 In this subsection we will study the ideal $\mathbf{I}_{reg}$ that defines the Segre embedding on $C_{reg}.$

Since we work over $C^{\alpha}$ rather than $\mathcal{H}^{\alpha}$, it is necessary to understand the pullback of the line bundle $L$ under $q^{\alpha}$.
We fix a coordinate $a$ on $C=A^1$ so get a coordinate system $(a^{i}_{j})_{i\in I, j\in \{1,\cdots,n_i\}}$ for $C^{\alpha}$. 

\begin{lemma}\label{pullback divisor}

The pullback of the divisors are
$$(q^{\alpha})^{*}\Delta_{ii}=div(\prod_{l\neq j}(a^{i}_l-a^{i}_j),  \text{ }(q^{\alpha})^{*}\Delta_{ii'}=div(\prod_{l, j}(a^{i}_l-a^{i'}_j)).$$
   
\end{lemma}
\begin{proof}
The second formula is clear. This first follows from the case when $I=\{1\}$ and $\alpha=2$ (recall $\alpha\in \mathbb{Z}^I$ is a dimension vector and when $I$ is a one-point set, $\alpha\in \mathbb{Z})$.  
\end{proof}
Denote the map replacing $\mathcal{I}$ by $L$ in (\ref{local I}) by $i^{L}_{\alpha}$.
We will describe the map $i_{\alpha}$ in (\ref{local I}) from the corresponding maps $i^{L}_{\alpha}$ of $L$. Note in lemma \ref{pre local of I}, we described the 
locality structure of $\mathcal{I}$ by morphisms between functor of points. 
Recall the map
$$\mathcal{H}\xleftarrow[]{pr_1} Gr(\mathcal{T})\xrightarrow[]{pr_2} \mathcal{H}$$
On the component $\mathcal{H}^{\alpha},$
it is 
$$\sqcup_{0\leq \beta  \leq \alpha}\mathcal{H^{\beta}}\xleftarrow[]{pr^{\alpha}_1} \sqcup_{0\leq \beta  \leq \alpha} Gr^{\beta}(\mathcal{T}^{\alpha})\xrightarrow[]{pr^{\alpha}_2} \mathcal{H}^{\alpha},$$
which is the disjoint union of 
\begin{equation}\label{corres}
\mathcal{H^{\beta}}\xleftarrow[]{pr^{\alpha,\beta}_1} Gr^{\beta}(\mathcal{T}^{\alpha})\xrightarrow[]{pr^{\alpha,\beta}_2} \mathcal{H}^{\alpha},
\end{equation}
Let $\mathcal{I}^{\alpha,\beta}=(pr^{\alpha,\beta}_2)_{*}(pr^{\alpha,\beta}_{1})^{*}L.$
For $\alpha$ we define an $I$-colored index set $S^{\alpha}=\sqcup_{i\in I} \{1,\cdots,n_i\}.$
For any subset $S\subset S^{\alpha},$ let $S_i$ be the $i$-th component of $S$ and let $|S|=(|S_i|)_{i\in I}.$
Define $C^S=\prod_{i\in I} \prod_{j\in S_i} C_{ij}$, where $C_{ij}=C$ so that $C^{\alpha}=C^{S^{\alpha}}.$ 
Now we will see what $(q^{\alpha})^{*}\mathcal{I}^{\alpha,\beta}$ is.
We introduce the following 
\begin{equation}\label{pullback of corres}
\mathcal{H}^{\beta} \xleftarrow[]{pr_3^{\alpha,\beta}}(q^{\alpha})^{*}Gr^{\beta}(\mathcal{T}^{\alpha})\xrightarrow[]{pr^{\alpha,\beta}_4:=(q^{\alpha})^{*}(pr_2^{\alpha,\beta})}C^{\alpha},
\end{equation}
where $pr_3^{\alpha,\beta}$ is the composition of maps from $(q^{\alpha})^{*}Gr^{\beta}(\mathcal{T}^{\alpha})$ to $Gr^{\beta}(\mathcal{T}^{\alpha})$ and  $Gr^{\beta}(\mathcal{T}^{\alpha})\xrightarrow[]{pr^{\alpha,\beta}_{1}}\mathcal{H}^{\beta}$.
We have 
 \begin{equation}\label{pullback of I}
(q^{\alpha})^{*}\mathcal{I}^{\alpha,\beta}\cong (pr^{\alpha,\beta}_4)_{*}(pr^{\alpha,\beta}_{3})^{*}(L).
 \end{equation}
Over $C^{\alpha}_{reg},$
$(q^{\alpha})^{*}Gr^{\beta}(\mathcal{T}^{\alpha}_{reg})\cong \sqcup_{S,|S|=|\beta|} C^{\alpha}_{reg}$
and 
the map $pr_1^{\alpha,\beta}$ factors through $\sqcup_{S,|S|=|\beta|}C^{\beta}_{reg}$ so (\ref{pullback of corres}) becomes
\begin{equation}\label{reason of decom of I regular}
\mathcal{H}_{reg}^{\beta}\xleftarrow[]{\sqcup q^{\beta}}\sqcup_{S,|S|=|\beta|}C^{\beta}_{reg} \xleftarrow[]{\sqcup pr^{\alpha,S}}\sqcup_{S,|S|=|\beta|}C^{\alpha}_{reg}\xrightarrow[]{\sqcup id}C^{\alpha}_{reg},    
\end{equation}
where $pr^{\alpha,S}$ is the projection of the $S$-components of $C^{\alpha}$ to $C^S$ composed with the identification map $C^{S}\cong C^{\beta}$ (and we still denote by the same notation $pr^{\alpha,S}$ when restricting on the regular part).
On the $S$-component, this is 
\begin{equation}\label{S component of reason of decom of I regular}
\mathcal{H}_{reg}^{\beta}\xleftarrow[]{q^{\beta}}C^{\beta}_{reg}\xleftarrow[]{pr^{\alpha,S}} C^{\alpha}_{reg}=C^{\alpha}_{reg}.
\end{equation}
 Denote the pullback of $L_{reg}^{\beta}$ under $pr^{\alpha,S}\circ q^{\beta}$ by $L_{reg}^{\alpha,S}$.
Hence $(q^{\alpha})^{*}\mathcal{I}^{\alpha}|_{reg}$
is the direct sum of line bundles $L_{reg}^{\alpha,S}$, i.e.
 \begin{equation}\label{decomp}
(q^{\alpha})^{*}\mathcal{I}^{\alpha}_{reg}\cong
\oplus_{0\leq \beta \leq  \alpha}  (q^{\alpha})^{*}\mathcal{I}^{\alpha,\beta}_{reg}
\cong \oplus_{0\leq \beta \leq \alpha} \oplus_{S,|S|=|\beta|} L^{\alpha,S}_{reg}=\oplus_{S,S\subset S^{\alpha}} L^{\alpha,S}_{reg}.
 \end{equation}
Here when $\beta=0$, $S=\emptyset$ and $L^{\alpha,\emptyset}$ is the pullback from $C^0$ to $C^{\alpha}$, which is the trivial line bundle (structure sheaf). 
Apply (\ref{decomp}) to the case where $\alpha=e_i$, since $q^{e_i}=id$ and $C^{e_i}_{reg}=C^{e_i}$, we have 
$$\mathcal{I}^{e_i} \cong L^{e_i,\emptyset} \oplus L^{e_i}.$$
 Pulling back under $C^{\alpha}\xrightarrow[]{pr_{ij}} C^{e_i}$, we have 
$$pr_{ij}^{*}\mathcal{I}^{e_i} \cong L^{\alpha,\emptyset} \oplus pr_{ij}^{*}L^{e_i}.$$
Now the locality isomorphism $i_{\alpha}$ of $\mathcal{I}$ in (\ref{local I}) (to be described) becomes
$$(\otimes_{i\in I}\otimes_{j\in (S_{\alpha})_i} (L^{\alpha,\emptyset} \oplus pr_{ij}^{*}L^{e_i}))|_{C^{\alpha}_{reg}}\xleftarrow[]{i_{\alpha}} \oplus_{S,S\subset S_{\alpha}} L^{\alpha,S}_{reg}.$$
Expanding the tensor product, since  $L^{\alpha,\emptyset}\otimes \mathcal{L}$ is the trivial line bundle, for any $S\subset S^{\alpha}$
the $S$-component of the map $i_{\alpha}$ is denoted by $i^{S}_{\alpha}$
\begin{equation}\label{localS}
    \otimes_{i\in I} \otimes_{j\in S_i} pr_{ij}^{*}L^{e_i}|_{C^{\alpha}_{reg}}\xleftarrow[]{i_\alpha^S}  L_{reg}^{\alpha,S}. 
    \end{equation}
Finally, we can describe the isomorphism $i^{S}_{\alpha}$.
It is the pullback under $C^{\beta}_{reg} \xleftarrow[]{pr^S} C^{\alpha}_{reg}$ for $\beta=|S|$
of an open part\footnote{In remark (\ref{locality1}), we decompose $\beta$ completely.} of the locality structure of $L^{\beta}$
\begin{equation}\label{local beta}
\otimes_{i\in I} \otimes_{j\in (S_{\beta})_i}pr^{*}_{ij}(L^{e_i})|_{C^{\beta}_{reg}}\xleftarrow[\cong]{i^{L}_{\beta}} (q^{\beta})^{*}L^{\beta}|_{C^{\beta}_{reg}}.
\end{equation}

It is not hard to see that $i_{\alpha}$ as described above are the isomorphisms that represent the locality isomorphisms in lemma \ref{pre local of I}.    

\subsubsection{Bases of global section $\Gamma(C^{\alpha},(q^{\alpha})^{*}\mathcal{I})$ as $\mathcal{O}(C^{\alpha})$-module (global bases) }\label{global basis for zastava}
For each $\beta$, we fix a basis $s^{\beta}\in \Gamma(C^{\beta},(q^{\beta})^{*}L^{\beta})$ adapted to the above local structure  as follows.
 First we fix a basis $s_i$ of global section of $L^{e_i}$ over $\mathcal{H}^{e_i}.$
 For any $\beta=(k_i)_{i\in I}\in \mathbb{N}[I],$
 let
\begin{equation}\label{def local  basis}
s^{\beta}=l(\beta)i_{\beta}(\otimes_{i\in I}\otimes_{j\in (S_{\beta})_i}pr^{*}_{ij}s_i)\in \Gamma(C^{\beta}_{reg},(q^{\beta})^{*}L^{\beta}|_{C_{reg}}),
\end{equation}
where \footnote{Here, we choose a basis according to $\kappa$, later, we will define $l_Q(\beta)$ depending on the quiver $Q$.}
 \begin{equation}\label{local fac}
 l(\beta)=\prod_{i<i', l\in (S_{\beta})_i,j\in (S_{\beta})_{i'}} (a^{i}_l-a^{i'}_j)^{\kappa_{ii'}} \prod_{i\in I,l\neq j,l,j\in(S_{\beta})_i} (a^{i}_l-a^{i}_j)^{\kappa_{ii}}
 \end{equation}
is in the fraction field $\mathcal{K}(\mathcal{H}^{\beta})$ of $\mathcal{H}^{\beta}.$
Here we still denote by $i_{\beta}$ the induced map between global sections of line bundles in (\ref{local beta}).
By lemma \ref{pullback divisor}, the section $s^{\beta}$ extend across the diagonal to a section of $C^{\beta}.$ Let $W^{\beta}$ be the group such that $C^{\beta}/W^{\beta}\cong \mathcal{H}^{\beta}$. It is easy to check that $s^{\beta}$ is $W^{\beta}$-invariant. Let $u^{\beta}\in \Gamma(\mathcal{H}^{\beta},L^{\beta})$ be the descent of $s^{\beta}$ under $q^{\beta}.$

Let $M^{\beta}$ be the global sections of $L^{\beta}$ over $\mathcal{H}^{\beta}$. It is a rank 1 free $\mathcal{O}(\mathcal{H}^{\beta})$-module generated by $u^{\beta}$. By (\ref{pullback of I}),
$$\Gamma(C^{\alpha},(q^{\alpha})^{*}\mathcal{I}^{\alpha,\beta})= M^{\beta}\otimes_{\mathcal{O}(\mathcal{H}^{\beta})} \mathcal{O}((q^{\alpha})^{*}Gr^{\beta}(\mathcal{T}^{\alpha}))$$ as $O(C^{\alpha})$-module. By choosing any basis $\{y^{\alpha,\beta}_p\}$ of $\mathcal{O}((q^{\alpha})^{*}Gr^{\beta}(\mathcal{T}^{\alpha}))$  as $\mathcal{O}(C^{\alpha})$-module, which is chosen as a basis of $\mathcal{O}(Gr^{\beta}(\mathcal{T}^{\alpha}))$ as $\mathcal{O}(\mathcal{H}^{\alpha})$-module, we have the basis $\{z^{\alpha,\beta}_p=u^{\beta} \otimes y_p\}$ of $\Gamma(C^{\alpha},(q^{\alpha})^{*}\mathcal{I}^{\alpha,\beta})$.
When $\alpha$ is fixed, we abbreviate $z^{\alpha,\beta}_p$ as $z^{\beta}_p$.
Such a basis is called a global basis.

\subsubsection{The basis of $\Gamma(C^{\alpha}_{reg},(q^{\alpha})^{*}\mathcal{I}_{reg})$ (local basis)}\label{local basis of zas}
From $1\in  \mathcal{O}((q^{\alpha})^{*}Gr^{\beta}(\mathcal{H}^{\alpha}))$, we get an element
$u^{\beta}\otimes 1$ in $$\Gamma(C^{\alpha},(q^{\alpha})^{*}\mathcal{I}^{\alpha,\beta})= M^{\beta}\otimes_{\mathcal{O}(\mathcal{H}^{\beta})} \mathcal{O}((q^{\alpha})^{*}Gr^{\beta}(\mathcal{T}^{\alpha})).$$ 
Restricting $u^{\beta}\otimes 1$ to $C^{\alpha}_{reg}$, by (\ref{reason of decom of I regular}), we have 
$s^{\beta}|_{reg}\otimes 1$ in  $$\Gamma(C_{reg}^{\alpha},(q^{\alpha})^{*}\mathcal{I}_{reg}^{\alpha,\beta})=
\bigoplus_{S,|S|=\beta}\Gamma(C_{reg}^{\alpha}, L^{\alpha,S}_{reg})=
\bigoplus_{S,|S|=\beta}\Gamma(C_{reg}^{\beta},(q^{\beta})^{*}L^{\beta}_{reg})\otimes_{\mathcal{O}(C^{\beta}_{reg})}{\mathcal{O}(C^{\alpha}_{reg})}.$$ 
In the decomposition in (\ref{decomp}),  
denote the $S$-component of $s^{\beta}|_{reg}\otimes 1$ by $s^{\alpha,S}\in \Gamma(C^{\alpha}_{reg},L^{\alpha,S}_{reg})$. We will refer to $s^{\alpha,S}$ as the local basis.
Now we fix an $\alpha$ and abbreviate $s^{\alpha,S}$ as $s^{S}.$
By our choices of bases $s^{\beta}$ in (\ref{def local  basis}) and the relation between (\ref{localS}) and (\ref{local beta}), we have
\begin{equation}\label{induced def of local}
s^{S}=l(S)i^{S}_{\alpha}(\otimes_{i\in I}\otimes_{j\in S_i}pr^{*}_{ij}s_i),
\end{equation}
 where  \begin{equation}
 l(S)=\prod_{i<i', l\in S_i,j\in S_{i'}} (a^{i}_l-a^{i'}_j)^{\kappa_{ii'}} \prod_{i\in I,l\neq j,l,j\in S_i}(a^{i}_l-a^{i}_j)^{\kappa_{ii}}
 \end{equation}
is the pullback of $l(\beta)$ under $C^{\beta}_{reg} \xleftarrow[]{pr^{\alpha,S}} C^{\alpha}_{reg}$ (see (\ref{S component of reason of decom of I regular})) and we still denote by $i^S_{\alpha}$ the induced map between global sections of line bundles in (\ref{localS}). By definition $\{s^S, S\subset S^{\alpha}\}$ generate the algebra $A_{0,reg}$ over $\mathcal{O}(C_{reg}^{\alpha})$.
 Now we can describe the ideal $\mathbf{I}_{reg}$ in terms of local basis $s^{S}$.
 \begin{lemma}\label{ideal in local}
The ideal $\mathbf{I}_{reg}$ is generated by 
$$\frac{s^A}{l(A)}\frac{s^B}{l(B)}=\frac{s^{A\cup B}}{l(A\cup B)}\frac{s^{A\cap B}}{l(A\cap B)}$$
 for all $A,B\subset S_{\alpha}.$
 \end{lemma}

 \begin{proof}
This is a standard result of Segre embedding. Since we cannot find a reference, we give a proof in the appendix.
 \end{proof}
 \subsubsection{Transition between a global basis $z^{\beta}_p$ and the local basis $s^S$ over the regular part}
First, we state a simple lemma about the module structure under localization.
\begin{lemma}\label{module loc structrue}
The following diagram commutes\\
\begin{tikzcd}
\mathcal{O}(Gr^{\beta}(\mathcal{T}^{\alpha}))\times \Gamma(C^{\alpha},(q^{\alpha})^{*}\mathcal{I}^{\alpha,\beta}) \arrow[r] \arrow[d, "loc\times loc"'] & \Gamma(C^{\alpha},(q^{\alpha})^{*}\mathcal{I}^{\alpha,\beta}) \arrow[d, "loc"] \\
\mathcal{O}(Gr^{\beta}(\mathcal{T}_{reg}^{\alpha}))\times \Gamma(C_{reg}^{\alpha},(q^{\alpha})^{*}\mathcal{I}_{reg}^{\alpha,\beta})\arrow[r]\arrow[d, "\cong"] & \Gamma(C_{reg}^{\alpha},(q^{\alpha})^{*}\mathcal{I}_{reg}^{\alpha,\beta})\arrow[d, "\cong"] \\
\bigoplus_{S,|S|=|\beta|} \mathcal{O}(C^{\alpha}_{reg})\times \bigoplus_{S,|S|=\beta}\Gamma(C_{reg}^{\alpha}, L^{\alpha,S}_{reg}) \arrow[r]&\bigoplus_{S,|S|=\beta}\Gamma(C_{reg}^{\alpha}, L^{\alpha,S}_{reg}),\\
\end{tikzcd}\\
where the horizontal maps are the ring actions on the modules.
\end{lemma}
\begin{proof}
 Clear from the definitions. 
\end{proof} 
Now we will study the following map in more details.
\begin{equation}\label{dec}
\mathcal{O}(Gr^{\beta}(\mathcal{T}^{\alpha}))\xrightarrow[]{|_{reg}}\mathcal{O}(Gr^{\beta}(\mathcal{T}_{reg}^{\alpha}))\cong \bigoplus_{S,S\subset S_{\alpha},|S|=|\beta|} \mathcal{O}(\mathcal{H}^{\alpha}_{reg}).
\end{equation}
Let $\beta=\sum_{i\in I} k_i i$. Denote $f^s()$ the $s$-th elementary symmetric functions on the variables inside the parenthesis.
\begin{theorem}\label{basis for local space}
a) $$\mathcal{O}(Gr^{\beta}(\mathcal{T}^{\alpha}))\cong 
\bigotimes_{i\in I} \mathcal{O}(Gr^{k_i}(\mathcal{T}^{n_i})).$$
b) The map $(pr_1,m)$ 
$$Gr^{k_i}(\mathcal{T}^{n_i})\xrightarrow[]{(pr_1,m)} \mathcal{H}^{k_i}\times \mathcal{H}^{n_i-k_i},$$ where $m((D',D)=D-D'$ is an isomorphism. \\
c) The projection $Gr^{k_i}(\mathcal{T}^{n_i})\xrightarrow[]{pr_2}\mathcal{H}^{n_i}$ is finite flat and 
$\mathcal{O}(Gr^{k_i}(\mathcal{T}^{n_i}))$ is an algebra extension over $\mathcal{O}(\mathcal{H}^{n_i })$ generated by formal variables $c^i_l,d^i_j,1\leq l \leq k_i,1\leq j \leq n_i-k_i$ subject to the relations 
$$<\sum_{1\leq l \leq k_i,l+j=s_i}c^i_l d^i_j-f^{s_i}(a^i_1,\cdots,a^i_{n_i})> $$ for all $s_i$ such that ${1\leq s_i\leq n_i}.$\\
d) Under the above isomorphism and (\ref{dec}), the image of $c^i_l$ in the $S_i$ component is $f^{l}(a^i_r)_{r\in S_i}$ and the image of $d^i_j$ in the $S_i$ component is $f^{j} (a^i_r)_{r\in \{1,\cdots,n_i\} \setminus S_i}$.

\end{theorem}
\begin{proof}
a) Since $$(Gr^{\beta}(\mathcal{T}^{\alpha}))\cong 
\prod_{i\in I} (Gr^{k_i}(\mathcal{T}^{n_i})).$$\\
b) Clear.\\
c) We omit the index $i$.
The map $pr_2$ being finite flat is already proved.
Since we have the map $Gr^k(\mathcal{T})\xrightarrow[]{pr_1} \mathcal{H}^k$, where $pr_1(D',D)=D$,  we can pull back functions on $\mathcal{H}^k$ to  $Gr^k(\mathcal{T})$.
Let $\mathcal{S}_h$ be the symmetric group of order $h$.
We have $f^l(a_1,\cdots,a_k)\in k[a_1,\cdots,a_k]^{\mathcal{S}_k} \cong \mathcal{O}(\mathcal{H}^k).$
Let their pullback under $pr_1$ be
$c_1,\cdots c_k$.
We also have a map $Gr^k(\mathcal{T})\xrightarrow[]{m} \mathcal{H}^{n-k}$
where $m(D',D)=D-D'$.
Similarly, we have $f^j(a_1,\cdots,a_{n-k})\in k[a_1,\cdots,a_{n-k}]^{\mathcal{S}_{n-k}} \cong \mathcal{O}(\mathcal{H}^{n-k}).$
and let their pullback under $m$ be
$d_1,\cdots d_{n-k}$.
 For $(D',D)\in Gr(\mathcal{T})$, where $D=\sum^{n}_{t=1} p_t, D'=\sum_{t\in S} p_t$, 
We have $$\sum_{l+j=s}c_ld_j((D',D))=\sum_{l+j=s}c_l(D)d_j(D-D')\\$$
$$=\sum_{l+j=s}f^l((a_r)_{r\in S})f^j((a_{t})_{t\in \{1,\cdots,a\}\setminus S})=f^s(a_1,\cdots,a_{n}).$$ Since $f^s(a_1,\cdots, a_{n})$ is $\mathcal{S}^{n}$ invariant, it can be viewed as functions on $\mathcal{H}_{n}$ of the same form pulling back by $pr_2$.\\
By part(b), the algebra $\mathcal{O}(Gr^{k}(\mathcal{T}^{n}))$ is already generated by $c_l,d_j$.
Over $\mathcal{H}_{reg}$ 
since $pr_2$ is finite flat, the degree can be seen over $\mathcal{H}_{reg}$ and from (\ref{dec}) it is  $\tbinom{n}{k}.$   
Now, these relations exhaust all from the well-known algebraic fact that the degree of the algebra extension is $\tbinom{n}{k}.$\\
d) Clear from the definition.
\end{proof}
Restricting to the regular part (as $\mathcal{O}(C_{reg})$-module), we still denote by $\{z^{\beta}_p\}$ as a basis of the $\mathcal{O}(C_{reg})$-module $\Gamma(C_{reg}^{\alpha},(q^{\alpha})^{*}\mathcal{I}_{reg}^{\alpha,\beta}).$

\section{Coulomb branch}
We briefly review the mathematical definition of the Coulomb branch\cite{braverman2018towards}.

Let $G$ be a reductive group and $N$ be a representation of $G$. Let $\mathcal{K}=\mathbb{C}((t))$ and $\mathcal{O}=\mathbb{C}[[t]]$.
We denote by $\mathcal{G}(G)=G_{\mathcal{K}}/G_{\mathcal{O}}$ the loop Grassmannian of $G$. Since $G$ acts on $N$, $G_{\mathcal{O}}$ acts on $N_{\mathcal{O}}$ and we get an associated bundle $G_{\mathcal{K}}\times_{G_{\mathcal{O}}} N_{\mathcal{O}}$. It has an embedding $i$ into $\mathcal{G}(G)\times N_{\mathcal{K}}$ given by $i(\overline{(g,v)})=(\overline{g},g\cdot v)$.
Let $\mathcal{R}_{(G,N)}$ be the preimage of $\mathcal{G}(G)\times N_{\mathcal{O}}$ under $i$.
In this paper, we will call Borel-Moore homology just homology.
In \cite{braverman2018towards}, they define an
algebra structure on the $G_{\mathcal{O}}$-equivariant homology of $\mathcal{R}$ and define the Coulomb branch of the pair $(G,N)$ to be the Spec of this algebra.

\subsection{Localization}
Let $T\subset G$ be a Cartan subgroup and $\mathfrak{t}$ be its Lie algebra.
 Let $\mathcal{R}_{T,N_T}$ be the variety of triples for the pair $(T,N_T)$, where $N_T$ is the $N$ considered as a representation of $T$.
Denote the localization of $\mathfrak{t}$ at all roots of $G$ and weights of $N_T$ by
$\mathfrak{t}_{reg}.$  
Let $\mathcal{R}_{T,N_T}\xrightarrow[]{\iota} \mathcal{R}$ be the embedding, the pushforward homomorphism $\iota_*$ 
\begin{equation}
H^{T_{\mathcal{O}}}_{*}(\mathcal{R}_{T,N_T})\xrightarrow[]{\iota_*} H^{T_{\mathcal{O}}}_{*}(\mathcal{R})
\label{iota}
\end{equation}
gives an algebra homomorphism, which becomes an isomorphism over $\mathfrak{t}_{reg}$\cite[Lemma (5.17)]{braverman2018towards}.\footnote{The notation $\mathfrak{t}_{reg}$ ignores its dependence on $N$.}

\subsection{Formula for $H^{T_{\mathcal{O}}}_{*}(\mathcal{R}_{T})$}
Now we fix a representation $N$ and abbreviate $\mathcal{R}_{N_T,T}$ by $\mathcal{R}_{T}$.
 We have $\mathcal{R}_{T}\xrightarrow[]{p} \mathcal{G}(T)$ .
 Since we consider the homology, we only count the reduced part of $\mathcal{G}(T)$ so we denote the reduced part of $\mathcal{G}(T)$ (still) by $\mathcal{G}(T)$.
 Then $\mathcal{G}(T)$ is disjoint union of points $t^{\chi}$,
  where $\chi$ are cocharacters of $T$.
  The fiber of $p$ over $t^{\chi}$ is a vector space, which we denote by $V_{\chi}$. 
  So we have $\mathcal{R}_{T}=\sqcup_{\chi} V_{\chi}.$
  Denote the $T_{\mathcal{O}}$-equivariant fundamental class of $V_{\chi}$ by $r^{\chi}$.
  Let $\xi_i$ be the characters of $T$ that appear in the representation $N$. Denote by $\xi_i(\chi)$ the pairing of $\xi_i$ and $\chi$.
For two integers $k,l$, let us set
$$d(k,l)=\left\{
\begin{aligned}
   & 0 & \text{if } k \text{ and }l \text{ have the same signs.}\\
   & min(|k|,|l|) &\text{if } k \text{ and }l \text{ have different signs.} \\
\end{aligned}
\right.$$
\begin{theorem}\cite[Theorem 4.1]{braverman2018towards}\label{multi for r^A} The algebra 
$H^{T_{\mathcal{O}}}_{*}(\mathcal{R}_{T})$ is generated by $r^{\chi}$ for all $\chi\in X_*(T)$ over $H^{T_{\mathcal{O}}}_{*}(pt)$.\\
For two cocharacters $\lambda$ and $\mu$,
the multiplication of $r^{\lambda}$ and $r^{\mu}$ is 
\begin{equation}\label{mult}
r^{\lambda}r^{\mu}=\prod^{n}_{i=1} \xi_{i}^{d(\xi_i(\lambda),\xi_i(\mu))}r^{\lambda+\mu}.
\end{equation}
\end{theorem}
\begin{rmk}\label{mult prop of C factors}
The coefficient before $r^{\lambda+\mu}$ depends on $\lambda,\mu$ and the representation $N$.
Denote the Grothendieck group of $T$ by $K^0(Rep(T))$ and the monomials of $X_{*}(T)$ by $Mon(X_{*}(T))$.
For fixed $\lambda,\mu$, 
the map $K^0(Rep(T))\xrightarrow[]{co^{\lambda,\mu}} Mon(X_{*}(T))$ given by   $co^{\lambda,\mu}(N)=\prod^{n}_{i=1} \xi_{i}^{d(\xi_i(\lambda),\xi_i(\mu))}$ is a homomorphism of monoids, i.e.
for $N=N_1\oplus N_2$ , we have
$$co^{\lambda,\mu}(N)=co^{\lambda,\mu}(N_1)co^{\lambda,\mu}(N_2).$$  
\end{rmk}

\subsection{Compactified Coulomb branch $\mathbf{M}^{\alpha}_{Q}$ of quiver gauge theory}

Here we consider a special case where the pair $(G,N)$ is given from a quiver $Q$.
Let $Q=(I,E)$ be the quiver where $I$ is the set of vertices and $E$ is the set of arrows.
Given $\alpha=\sum_{i\in I}n_i i\in \mathbb{N}[I]$, we view it as a dimension vector $\alpha=(n_i)_{i\in I}$.
Let $V_i=k^{n_i}$ and $V=(V_i)_{i\in I}$ is an $I$- graded vector space. 
Let $N=\oplus _{i\xrightarrow[]{} j} Hom(V_i,V_j)$
and $G=\prod_i GL(V_i)$ act on $N$ by conjugation.

Following \cite{braverman2018towards} 3(ii), for any vector space $U$, we define $\mathcal{G}^{+}(GL(U))\subset \mathcal{G}((GL(U))$ as the moduli space of vector bundles $\mathcal{U}$ on the formal disc $D$ with a trivialization $\sigma:\mathcal{U} _{D^{*}}\cong U\otimes \mathcal{O}_{D^{*}}$ on the punctured disc that extends through the puncture as an embedding $\sigma:\mathcal{U} _{D}\hookrightarrow U\otimes \mathcal{O}_{D^{*}}$.
Let $\mathcal{G}^{+}_{GL(V)}=\prod_{i\in I} \mathcal{G}^{+}_{GL(V_i)}$. Define $\mathcal{R}^{+}$ as the preimage of $\mathcal{G}^{+}_{GL(V)}\subset \mathcal{G}_{GL(V)}$ under $\mathcal{R}\xrightarrow[]{} \mathcal{G}_{GL(V)}$. The homology group $H^{GL(V)_{\mathcal{O}}}_{*}(\mathcal{R}^{+})$ forms a convolution subalgebra of $H^{GL(V)_{\mathcal{O}}}_{*}(\mathcal{R})$. 

In general, given a $\mathbb{Z}^{\geq 0}$ filtration of an algebra $R$, $0\subset F_1 \subset F_2 \cdots \subset R$, we can get a $\mathbb{Z}^{\geq 0} $-graded algebra, the Rees algebra $Rees_F(R):=\oplus_i F_it^i$, where $t$ is a formal variable. Then $Proj(Rees_F(R))$ is a compactification of $Spec(R)$.
When the filtration $F$ is fixed, we often omit $F$ from the notation $Rees_F(R)$.

Now we give the algebra\footnote{Since we fix the dimension vector $\alpha$, we omit $\alpha$ for the notation $\mathcal{A}$.} $\mathcal{A}\stackrel{\mathrm{def}}{=}H_{*}^{GL(V)_{\mathcal{O}}}(\mathcal{R}^{+})$ a $\mathbb{Z}_{\geq 0}$ filtration, which is the pullback of the filtration in remark 3.7 \cite{braverman2018towards} under the diagonal embedding of $\mathbb{Z}_{\geq 0}\xrightarrow[]{}\mathbb{Z}_{\geq 0}^{I}$, as follows.
Recall $n_i=dim V_i$.
The $GL(V)_{\mathcal{O}}$-orbits in $Gr^+_{GL(V)}$ are numbered by $I$-colored partitions $(\lambda^{(i)})_{i\in I}$, $\lambda^{(i)}=(\lambda^{(i)}_1\geq \lambda^{(i)}_2\geq \cdots ),  $ such that the number of parts $l(\lambda^{(i)})\leq n_i$.
Given a dimension vector $(m_i)_{i\in I}$, we define a closed $GL(V)_{\mathcal{O}}$-invariant subvariety $\overline{Gr_{GL(V)}}^{+,(m_i)}\subset 
Gr^+_{GL(V)}$ as the union of orbits $Gr^{(\lambda^{(i)})_{i\in I}}_{GL(V)}$ such that $\lambda^{(i)}_1\leq m_i$ for any $i\in I$.
For $m\in \mathbb{N}$, we define $\mathcal{R}^{+}_{\leq m}\subset \mathcal{R}^{+}$ as the preimage of $\overline{Gr}_{GL(V)}^{+,(m,\cdots,m)}$
under $\mathcal{R}^{+}\xrightarrow[]{} Gr^{+}_{GL(V)}$. This gives a $\mathbb{Z}^{\geq 0}$ filtration $F_m=H_{*}^{GL(V)_{\mathcal{O}}}(\mathcal{R}^{+}_{\leq m})$ of the algebra $\mathcal{A}.$

Define $\mathbf{M}_{Q}^{\alpha}\stackrel{\text{def}}{=}Proj (Rees(\mathcal{A}))$ and call $\mathbf{M}_{Q}^{\alpha}$ the compactified Coulomb branch.

\subsection{Embedding of $\mathbf{M}_{Q}^{\alpha}$ into $\mathbb{P}_{\mathcal{H}}(H^{G_{\mathcal{O}}}(\mathcal{R}^{+}_{\leq 1})^{*})$}
We fix an ordered basis of $V.$ Let $T_i$ be corresponding diagonal subgroup of $GL(V_i)$ and $\mathfrak{t_i}$ be its Lie algebra. Let $W_i$ be the Weyl group of $GL(V_i).$ 
Denote by $(a^i_j)_{1\leq j \leq n_i}$ the corresponding standard basis of $\mathfrak{t_i}$. Notice that in the last section, we already denote certain generators of $\mathcal{O}(C^{\alpha})$ by $a^i_j$. 
Here we identify 
$H^{*}_{G_{\mathcal{O}}}(pt)\hookrightarrow H^{*}_{T_{\mathcal{O}}}(pt)$
with $\mathcal{O}(\mathcal{H}^{\alpha})\hookrightarrow \mathcal{O}(C^{\alpha})$ via $Sym((\mathfrak{t}/W)^{*})\hookrightarrow Sym(\mathfrak{t}^{*})$.
The algebra $\oplus_m H^{G_{\mathcal{O}}}(R^{+}_{\leq m})t^m$ is a module over $H^{*}_{G_{\mathcal{O}}}(pt)\cong H^{*}_{G}(pt)$ so we have
a projection
\begin{equation}\label{projection}
    \mathbf{M}_{Q}^{\alpha}\xrightarrow[]{} \mathcal{H}^{\alpha}.
\end{equation}

\begin{lemma}\label{generate}
As a $H^{*}_{G_{\mathcal{O}}}(pt)$-algebra, the Rees algebra
$\oplus_m H_{*}^{G_{\mathcal{O}}}(\mathcal{R}^{+}_{\leq m})t^m$ is generated by $ H_{*}^{G_{\mathcal{O}}}(\mathcal{R}^{+}_{\leq 1})t$.
\end{lemma}

\begin{proof}
We recall some notations in \cite{braverman2019coulomb}.
Let $\mathcal{R}_{\lambda} = \mathcal{R}\cap \pi^{-1}(Gr^{\lambda})$ ($\mathcal{R}_{\leq \lambda}= \mathcal{R}\cap \pi^{-1}(\overline{Gr^{\lambda})}$) be the restriction of $\mathcal{R}$ on the $G_\mathcal{O}$-orbit $Gr^{\lambda}$ ($G_\mathcal{O}$-orbit closure $\overline{Gr^{\lambda}}$, resp).
 Let $\mathcal{R}_{< \lambda}$ be the complement $\mathcal{R}_{\leq \lambda}\setminus \mathcal{R}_{\lambda}$.
It is a closed subvariety.
Lemma 6.2 in \cite{braverman2018towards}
says that the Mayer-Vietoris
sequence splits into short exact sequences
$$0\xrightarrow[]{} H^{G_{\mathcal{O}}}_{*}(\mathcal{R}_{< \lambda}) \xrightarrow[]{} H^{G_{\mathcal{O}}}_{*}(\mathcal{R}_{\leq \lambda}) \xrightarrow[]{} H^{G_{\mathcal{O}}}_{*}(\mathcal{R}_{\lambda})\xrightarrow[]{} 0.$$
Moreover, as $H^{*}_{G_{\mathcal{O}}}(pt)$-module, this exact sequence splits canonically so we have 
\begin{equation}
H^{G_{\mathcal{O}}}_{*}(\mathcal{R}_{\leq \lambda})=H^{G_{\mathcal{O}}}_{*}(\mathcal{R}_{< \lambda})\oplus H^{G_{\mathcal{O}}}_{*}(\mathcal{R}_{\lambda}).\label{sum}
\end{equation}
It suffices to show that any $a\in H^{G_{\mathcal{O}}}(\mathcal{R}^{+}_{\leq m})t^m$ is generated by $H^{G_{\mathcal{O}}}(\mathcal{R}^{+}_{\leq 1})t$.
Let $\Lambda(m)$ be the set of all maximal $\lambda$ such that $\lambda^{(i)}_1\leq m$.
By definition of $\mathcal{R}^{+}_{\leq m}$ and $\overline{Gr_{\lambda}}=\cup_{\mu:\mu\leq \lambda}Gr_{\mu}$,
we have  $\mathcal{R}^{+}_{\leq m}=\sqcup_{ \lambda\in \Lambda(m) }\mathcal{R}^{+}_{\leq \lambda}$.
It suffices to prove the claim for any $a\in H^{G_{\mathcal{O}}}(\mathcal{R}^{+}_{\leq \lambda})t^m$ for $\lambda \in \Lambda(m)$.
We prove this by induction on $\lambda$. Suppose for any $\mu< \lambda$, the theorem holds.
By formula (\ref{sum}) and induction hypothesis, it suffices to show when $a\in H^{G_{\mathcal{O}}}_{*}(\mathcal{R}_{\lambda})$ for some $\lambda$.
The second paragraph after proposition 6.1 \cite{braverman2019coulomb} says
$$H_{*}^{G_{\mathcal{O}}}(R_{\lambda})\cong H^{*}_{Stab_{G}(\lambda)}(pt) \cap [\mathcal{R}_{\lambda}]\cong  \mathbb{C}[\mathbf{t}]^{W_{\lambda}}
[\mathcal{R}_{\lambda}],$$
where $\cap$ is the cap product, $W_{\lambda}$ is the stabilizer of $\lambda$ in the Weyl group $W$ and $[\mathcal{R}_{\lambda}]$ is the $G_{\mathcal{O}}$-equivariant fundamental class of $\mathcal{R}_{\lambda}$.
So it suffices to prove the case $a=[\mathcal{R}_{\lambda}]t^m.$
Recall we denote $H^{GL(V)_{\mathcal{O}}}_{*}(\mathcal{R}^+)$ by $\mathcal{A}$.
In the proof of \cite{braverman2019coulomb} proposition 6.8, regarding $[\mathcal{R}_{\lambda}]$ as an element in $gr\mathcal{A}$,
we have 
 $$[\mathcal{R}_{\lambda}][\mathcal{R}_{\mu}] = [\mathcal{R}_{\lambda+\mu}]$$ in $gr\mathcal{A}$ when $\lambda, \mu$ are in the same
"generalized Weyl chamber". In this case, any weight of $N$ is a root of $G$, so their "generalized Weyl chamber" is the same as the usual Weyl chamber.
Let $\varpi_{\beta}$ be the fundamental weights of $GL(V)$. Here, the fundamental weights of $GL_n$ consist of all $\varpi_{k}=(1,1\cdots,1,0,0\cdots,0)$ for $1\leq k\leq n$ where $k$ is the number of $1$'s. The fundamental weights of the product of $GL(V_i)$'s consist of the disjoint union of fundamental weights of $GL(V_i)$'s. 
Since $\mathcal{R}_{\beta}\subset \mathcal{R}_{\leq 1}$, we have $H^{G_{\mathcal{O}}}(\mathcal{R}_{\beta})t\subset Rees\mathcal{A}.$
We can assume $max_i \lambda^{(i)}_1=m$.
It is easy to see that we can choose $m$
fundamental weights $\beta_i$ such that $\sum_i \beta_i=\lambda$.
So in $\mathcal{A}$, we have $[\mathcal{R}_{\beta_1}][\mathcal{R}_{\beta_2}]\cdots [\mathcal{R}_{\beta_m}]-[\mathcal{R}_{\lambda}]\in H^{G_{\mathcal{O}}}(\mathcal{R}_{<\lambda})$. In $Rees\mathcal{A}$, we have 
$[\mathcal{R}_{\beta_1}]t[\mathcal{R}_{\beta_2}]t\cdots [\mathcal{R}_{\beta_m}]t-[\mathcal{R}_{\lambda}]t^m\in H_{*}^{G_{\mathcal{O}}}(\mathcal{R}_{<\lambda})t^m$, which is generated by $H_{*}^{G_{\mathcal{O}}}(\mathcal{R}_{\leq 1})$ by induction. Hence we claim $[\mathcal{R}_{\lambda}]t^m$ is generated by $H_{*}^{G_{\mathcal{O}}}(\mathcal{R}_{\leq 1})$.

\end{proof}

\begin{cor} \label{embed M}We have an embedding
$$\mathbf{M}^{\alpha}\xhookrightarrow[]{j_\alpha} 
 \mathbb{P}_{\mathcal{H}^{\alpha}}(H_{*}^{G}(\mathcal{R}^{+}_{\leq 1})^{*}),$$ 
 where $H_{*}^{G}(\mathcal{R}^{+}_{\leq 1})^{*}$ is understood as a vector bundle over $Spec(H^{*}_{G}(pt))\cong \mathcal{H}^{\alpha}.$
\end{cor}

\begin{proof}
By lemma \ref{generate}, we have the surjection of graded algebras
$$\mathcal{S}(H_{*}^{G_{\mathcal{O}}}(\mathcal{R}^{+}_{\leq 1})t)\twoheadrightarrow \oplus_m H_{*}^{G_{\mathcal{O}}}(\mathcal{R}^{+}_{\leq m})t^m.$$
Taking Proj, we get the embedding.
\end{proof}
\subsection{Bases of $H^{G_{\mathcal{O}}}_{*}(\mathcal{R}^{+}_{\leq 1})$ as $H_{G_{\mathcal{O}}}^{*}(pt)\cong \mathcal{O}(\mathcal{H}^{\alpha})$-module}\label{coloumb global basis}
\begin{lemma}\label{bundle over grass}

\begin{equation}\label{R1}
H^{G_{\mathcal{O}}}_{*}(\mathcal{R}^{+}_{\leq 1})\cong \bigoplus_{(k_i)_{i\in I}} \bigotimes_{i\in I} H^{GL(V_i)}_{*}(Gr(n_i,k_i)). 
\end{equation}

\end{lemma}
\begin{proof}    
Let $\{\varpi^i_{k_i},1\leq k_i \leq n_i\}$ be all fundamental weights of $GL(V_i)$, where $\varpi^i_{k_i}=(1,1\cdots,1,0,0\cdots,0)$, $k_i$=number of $1$'s. Now we allow $k_i=0$, denote $\varpi^i_{0}=(0,\cdots,0)$ and let $\varpi_{(k_i)_{i\in I}}\stackrel{def}{=}(\varpi^i_{k_i})_{i\in I}$.
Recall $V=\oplus_{i\in I}V_i$ with dimension vector $\alpha=(n_i)_{i\in I}.$
For the set $\{\varpi_{(k_i)_{i\in I}},0\leq k_i \leq n_i,i\in I\}$,
 we have
$$\mathcal{R}^{+}_{\leq 1}=\bigsqcup_{0\leq k_i \leq n_i,i\in I}
\mathcal{R}_{\varpi_{(k_i)_{i\in I}}}.$$
For each component of $\mathcal{R}^{+}_{\leq 1}$, $\mathcal{R}_{\varpi_{(k_i)_{i\in I}}}$ is a vector bundle over the $GL(V)_{\mathcal{O}}$-orbit $\mathcal{G}_{\varpi_{(k_i)_{i\in I}}}=\prod_i Gr(n_i,k_i)$ so their homologies are isomorphic (without grading).
\end{proof}
Now we fix a dimension vector  $\beta=(k_i)_{i\in I}$.
By Poincare duality, $H_{G}^{*}(\mathcal{R}_{\varpi_{\beta}})\xrightarrow[\sim]{p}H^{G}_{*}(\mathcal{R}_{\varpi_{\beta}})$, where\footnote{again we ignore the grading} $p(c)=c\cap [\mathcal{R}_{\varpi_{\beta}}]^{G} $
and $[\mathcal{R}_{\varpi_{\beta}}]^{G}$ is the $G$-equivariant fundamental class of $\mathcal{R}_{\varpi_{\beta}}$. 
The homology group $H^{G}_{*}(\mathcal{R}_{\varpi_{\beta}})$  is a free rank 1 module over the ring $H_{G}^{*}(\prod_{i\in I} Gr(n_i,k_i))\cong \otimes _{i\in I} H_{GL(V_i)}^{*}Gr(n_i,k_i).$ 
For the cohomology of Grassmannian, we have a well-known result. Since the cohomology of the product of Grassmannians is just the tensor product of cohomologies of Grassmannians, for brevity, in lemma \ref{coh of grass}, we set $I=\{1\}$ so $GL(V)=GL_n$ and we abbreviate $a^1_{j}$ by $a_{j}$ for $1\leq j\leq n.$
Let $S$ be the tautological bundle over $Gr(n,k)$ and $Q$ be the quotient bundle. Denote $c_i$ the $i$-th $G$-equivariant Chern class and $c$ the total Chern class. Recall that the notation $f^s$ is the elementary symmetric functions, introduced before theorem \ref{basis for local space}.  
\begin{lemma}\label{coh of grass}
$$H_{GL(V)}^{*}Gr(n,k)\cong H_{GL(V)}^{*}(pt)[c_l(S),c_j(Q)]/\mathrm{I},$$
where the ideal $\mathrm{I}$ is generated by 
$$\sum_{1\leq l \leq k,l+j=s}c_l(S)c_j(Q)-f^{s}(a_1,\cdots,a_{n}) $$ for all $s$ such that ${1\leq s\leq n}.$
\end{lemma}
\begin{proof}

 We have an exact sequence $0\xrightarrow[]{} S\xrightarrow[]{} V \xrightarrow[]{} Q \xrightarrow[]{} 0$. It is well-known that 
as an $H_{GL(V)}^{*}(pt)$-algebra, $H^*_{T}(Gr(n,k)$ has generators $c_1(S),\cdots,c_k(S),c_1(Q),\cdots  c_{n-k}(Q)$.
From the exact sequence, we have $c(S)c(Q)=c(V)$.
$c(V)=\prod_{\chi} c(V_{\chi})=\prod_{\chi}(1+\chi)=\prod_i(1+a_i)$, where $\chi$ are characters of $T$ and $V_{\chi}$ is the  $\chi$-weight space of $V$. Plug in $c=\sum_i c^i$ into $S,Q$ and expand $c(S)c(Q)=c(V)$. Comparing the degree $s$ part, we get the relation in the lemma.

\end{proof}
\begin{rmk}\label{coh of grass T}
By the same argument, we get the isomorphism for $T$-equivariant cohomology.
$$H_{T}^{*}(Gr(n,k))\cong H_{T}^{*}(pt)[c_l^{T}(S),c_j^{T}(Q)]/I,$$
for the ideal $I$ with the same generators from the lemma.
\end{rmk}

\subsection{Pulling back $\mathbf{M}^{\alpha}$ under $C^{\alpha}\xrightarrow[]{q^{\alpha}} \mathcal{H}^{\alpha}$}
To apply localization theory in (3.1), we consider $H^{T_{\mathcal{O}}}_{*}(\mathcal{R})$ which is an algebra over $H_{*}^{T_{\mathcal{O}}}(pt)$.
We identify $Spec(H_{*}^{T_{\mathcal{O}}}(pt))$ with $C^{\alpha}$.
\begin{lemma}
The pullback of $Spec(H^{G_{\mathcal{O}}}_{*}(\mathcal{R}))$, $\textsf{M}^{\alpha}$ and $\mathbf{M}^{\alpha}$ under the quotient map $C^{\alpha}\xrightarrow[]{q^{\alpha}} \mathcal{H}^{\alpha}$ is 
$Spec(H^{T_{\mathcal{O}}}_{*}(\mathcal{R}))$,
$Spec(H^{T_{\mathcal{O}}}_{*}(\mathcal{R}^{+}))$ and  $Proj(\oplus_m H_{*}^{T_{\mathcal{O}}}(\mathcal{R}^{+}_{\leq m})t^m)$
\end{lemma}
\begin{proof}
 The following is a Cartesian diagram
 \begin{tikzcd}
   H^{T_{\mathcal{O}}}_{*}(?) \arrow{r}{} \arrow{d}{}& H^{G_{\mathcal{O}}}_{*}(?) \arrow{d}{} \\
    H^{T_{\mathcal{O}}}_{*}(pt)\arrow{r}{} &  H^{G_{\mathcal{O}}}_{*}(pt)
\end{tikzcd}
where $?$ is $\mathcal{R}$, $\mathcal{R}^{+}$ and $\mathcal{R}_{\leq m}^{+}$.
\end{proof}
\begin{lemma}
  For the quiver gauge theory case, under the identification of $\mathfrak{t}$ and $C^{\alpha}$, 
  the pushforward homomorphism $\iota_*$ for the inclusion $\mathcal{R}_{T,N_T}\xrightarrow[]{\iota}\mathcal{R}$
$$H^{T_{\mathcal{O}}}_{*}(\mathcal{R}_{T,N_T})\xrightarrow[]{\iota_*} H^{T_{\mathcal{O}}}_{*}(\mathcal{R})$$
becomes an isomorphism over $C^{\alpha}_{reg}$.
\end{lemma}

\begin{proof}
In (\ref{iota}), 
it says the pushforward homomorphism $\iota_*$ becomes an isomorphism over $\mathfrak{t}_{reg}$. 
For each direct summand $Hom(V_i,V_{i'})$ in $N$,
the localization inverts all $a^{i}_l-a^{i'}_j$ for $1\leq l \leq dimV_i, 1\leq j \leq dimV_{i'}$.
It is clear that the localization of $C^{\alpha}$ to $C^{\alpha}_{reg}$ includes 
all roots of $G$ and weights
of $N_T$.
\end{proof}
We define $\mathcal{R}^{+}_{T,N_T}$ and its filtration $(\mathcal{R}^{+}_{T,N_T})_{\leq m}$ as the pullback under $\iota$ of  $\mathcal{R}^{+}$ and its filtration $\mathcal{R}_{\leq m}^{+}$. Then we get the Rees algebra $\oplus_m H_{*}^{T_{\mathcal{O}}}((\mathcal{R}^{+}_{T,N_T})_{\leq m})t^m$ of $H^{T_{\mathcal{O}}}_{*}(\mathcal{R}_{T,N_T}).$
For any space $X$, we denote
\begin{equation}\label{localization natation}
\mathcal{H}^{T}(X)\stackrel{def}{=}H^{T}(X)\otimes_{\mathcal{O}(C^{\alpha})} \mathcal{O}(C^{\alpha})_{reg}.
\end{equation}
For example, by this convention, the localized Rees algebra of  $\mathcal{R}^{+}_{T,N_T}$ is denoted by $\oplus_m \mathcal{H}^{T_{\mathcal{O}}}((\mathcal{R}_{T,N_T})^{+}_{\leq m})t^m,$ 



We fix some notations.
In our case $\mathcal{G}(T)$ is the disjoint union of points $t^{\chi}=(t_i)^{\chi^{i}}_{i\in I}$
  where $\chi=(\chi^{i})_{i\in I}$ is a cocharacter of $T=\prod_{i\in I} T_i$ and $\chi^{i}$ is a  cocharactor of $T_i$.
  By the standard isomorphism $T_i\cong (G_m)^{n_i}$ we can write $\chi^i=(\chi_{j}^i)_{1\leq j \leq n_i}$.

\subsubsection{Generators of the localized Rees algebra $\oplus_m \mathcal{H}^{T_{\mathcal{O}}}((\mathcal{R}_{T,N_T})^{+}_{\leq m})t^m$}

From the definition of the filtration of $\mathcal{R}_{T,N_T}$ , it is easy to see that

$$(\mathcal{R}_{T,N_T})^{+}_{\leq m}=\bigsqcup_{ 0\leq \chi^i_j \leq m \text{ for any } i,j } V_\chi.$$
\begin{lemma}\label{local generators}
The localized Rees algebra $\oplus_m \mathcal{H}^{T_{\mathcal{O}}}((\mathcal{R}_{T,N_T})^{+}_{\leq m})t^m$ is generated by $\mathcal{H}^{T_{\mathcal{O}}}((\mathcal{R}_{T,N_T})^{+}_{\leq 1})t.$
\end{lemma}
\begin{proof}
This follows from the formula (\ref{mult}) in theorem \ref{multi for r^A}. 
\end{proof}

\subsection{$H^{T_{\mathcal{O}}}_{*}((\mathcal{R}_{T,N_T}^{+})_{\leq 1})\xrightarrow[]{\iota_*} H^{T_{\mathcal{O}}}_{*}(\mathcal{R}^{+}_{\leq 1})$ over $C_{reg}$}\label{local basis of coloumb}

In the next lemma, we will study the map
$$H^{T_{\mathcal{O}}}_{*}(\mathcal{R}^{+}_{\leq 1})\xhookrightarrow[]{}\mathcal{H}^{T_{\mathcal{O}}}_{*}(\mathcal{R}^{+}_{\leq 1})\xrightarrow[]{(\iota_*)^{-1}} \mathcal{H}^{T_{\mathcal{O}}}_{*}(\mathcal{R}_{T,N_T})^{+}_{\leq 1})\cong \bigoplus_{\chi|0\leq \chi^i_j \leq 1 \text{ for any } i,j } \mathcal{H}^{T_{\mathcal{O}}}(V_\chi).$$

Since the preimage of $\mathcal{R}_{\varpi_{\beta}}$ under $\iota$ is $\sqcup_{\chi\in W\varpi_{\beta}}V_\chi$, restricting the above map to $H_{*}^{T_{\mathcal{O}}}(\mathcal{R}_{\varpi_{\beta}})$ we get 
\begin{equation}
H^{T_{\mathcal{O}}}_{*}(\mathcal{R}_{\varpi_{\beta}})\xhookrightarrow[]{}\bigoplus_{\chi\in W\varpi_{\beta}} H_{*}^{T_{\mathcal{O}}}(V_\chi).\label{inj}
\end{equation}

For each $\chi\in W\varpi_{\beta}$, let $x^{\chi}$ be the $\chi$ component of the image under (\ref{inj}) of the fundamental class of $[\mathcal{R}_{\varpi_{\beta}}]$.
For any $\chi$ such that $0\leq \chi^i_j \leq 1$ for all $i,j$, we define a set $S(\chi)=\sqcup_{i\in I}S(\chi_i)$ where $S(\chi_i)=\{j|\chi^i_j=1\}.$
It is clear such $\chi$ is in bijection with $S$ such that $S\subset S^{\alpha},|S|=\beta$. So given $S$,  we also use the notation $\chi(S)$, meaning the $\chi$ such that $S(\chi)$ is the given set $S$.
We denote the set $\{1,\cdots,n_i\}$ by $[n_i]$.
Recall we denote the fundamental class $[V_{\chi}]$ by $r^{\chi}\in H_{*}^{T_{\mathcal{O}}}(V_{\chi})$.

\begin{lemma}\label{euler factor}
a)$$x^{\chi}=(Eu_{\chi})^{-1} r^{\chi},$$
where $Eu_{\chi}$ is the $T$-equivariant Euler class of $N_{\mathcal{R}_{\varpi_{\beta}}/V_{\chi}}$ at the point $t^{\chi}$. Here   $N_{\mathcal{R}_{\varpi_{\beta}}/V_{\chi}}$ is the normal bundle of $V_{\chi}$ in $\mathcal{R}_{\varpi_{\beta}}.$\\
b)$$Eu_{\chi}=\prod_{i\in I}\prod_{j\in S(\chi_i),l\in [n_i]\setminus S(\chi_i)}(a^i_j-a^i_l).$$ 

\end{lemma}

\begin{proof}
a) 
This is just localization theory.
b)
Let $(e^i_j)_{i\in I,j\in [n_i]}$ be the $T$-eigenvectors of $V=\oplus_{i \in I}V_i.$
Since the fiber of the projection $\mathcal{R}_{\varpi_{\beta}}\xrightarrow[]{} \mathcal{G}(GL(V))_{\varpi_{\beta}}=\prod_{i\in I}Gr(n_i,k_i)) $ at $t^{\chi}$ is $V_{\chi}$,
we have 
$$(N_{\mathcal{R}_{\varpi_{\beta}}/V_{\chi}})_{t^{\chi}}=T_{[\oplus_{i\in I}\oplus_{j\in S(\chi_i)}ke^i_j]}(\prod_{i\in I}Gr(n_i,k_i)),$$
where $[\oplus_{i\in I}\oplus_{j\in S(\chi_i)}ke^i_j]$ is the point $t^{\chi}$ viewed as a point in $\prod_{i\in I}Gr(n_i,k_i)$ (i.e., a vector subspace in $V$).
Now the formula follows from a standard computation about weight spaces of tangent spaces of a partial flag variety.
\end{proof}
\begin{cor}\label{generators x chi}
 The localized Rees algebra $\oplus_m \mathcal{H}^{T_{\mathcal{O}}}((\mathcal{R}_{T,N_T})^{+}_{\leq m})t^m$ is generated by $x^{\chi(S)}t,S\subset S^{\alpha}$.   
\end{cor}
\begin{proof}
Follows from lemma \ref{local generators} and lemma \ref{euler factor} part (a).    
\end{proof}
\section{Identification of $\mathbf{M}^{\alpha}$ and $Z^{\alpha}$}
For a quiver $Q$, we forget the direction of arrows and define an integer matrix  $\kappa(Q)$ as $\kappa_{ii}=1-$ number of self-loop of $i$ 
  and $\kappa_{ij}$=-number of arrow between $i$ and $j$.  
For a symmetric integral matrix $\kappa$, if $\kappa_{ii}\leq 1$ for all $i\in I$ and $\kappa_{ij}\leq 0$ for all $i,j\in I, i\neq j$,
we say $\kappa$ is of quiver-type, since in this case there exists a quiver $Q$ such that $\kappa=\kappa(Q).$

\begin{rmk}[Different conventions about the matrix $\kappa(Q)$]
For a diagram $Q$ (possibly with self-loops), define the Cartan matrix $C(Q)$ associated with $Q$ as $C_{ii}=2-2$ number of self-loop of $i$ 
  and $C_{ij}$=-number of arrow between $i$ and $j$. 
So the relation is $C_{ii}=2\kappa_{ii}, C_{ij}=\kappa_{ij}.$
For $Q$ of finite ADE type,
let $G_Q$ be the simply-connected simple algebraic group corresponding to $Q$. As mentioned in the introduction,
in \cite{mirkovic2021loop,mirkovic}, Mirkovic proved that $Z^{\alpha}_{\kappa(Q)}$\footnote{In our notation} is isomorphic to the Zastava space $Z^{\alpha}_{G_Q}$. The notation $\kappa$
in \cite{mirkovic2021loop,mirkovic} is $C(Q)$ here.
\end{rmk}

\begin{theorem}

\label{main}
For a quiver $Q$, 
the Compactified Coulomb branch 
$\mathbf{M}_{Q}^{\alpha}$ is canonically isomorphic to the local projective space $Z_{\kappa(Q)}^{\alpha}.$ 
\end{theorem}

Again, we will prove it after pullback under $C^{\alpha}\xrightarrow[]{a}\mathcal{H}^{\alpha}$. We list what we have.
\begin{center}
\begin{tabular}{ |c|c|c| } 
 \hline
        & $(q^{\alpha})^{*}\mathbf{M}^{\alpha}$ & $(q^{\alpha})^{*}Z^{\alpha}$ \\ 
         \hline
 embedding & $(q^{\alpha})^{*}\mathbf{M}^{\alpha}\xhookrightarrow[]{j_\alpha} 
 \mathbb{P}_{C^{\alpha}}(H_{*}^{T}(\mathcal{R}^{+}_{\leq 1})^{*})$ (a) &$(q^{\alpha})^{*}Z^{\alpha}\xhookrightarrow[]{i_\alpha}  \mathbb{P}_{C^{\alpha}}((q^{\alpha})^{*}V^{\alpha})$ (b) \\ 
  \hline
 over regular part $C^{\alpha}_{reg}$ & product of $\mathbb{P}^1$ (c)&  product of $\mathbb{P}^1$ (d) \\ 
 \hline
 global basis   & basis of $H_{*}^{T}(\mathcal{R}^{+}_{\leq 1})$ (e)&   basis of $\Gamma(C^{\alpha},(q^{\alpha})^{*}\mathcal{I}^{\alpha})$ (f)\\ 
   \hline
 local basis   & $s^S$ (section \ref{local basis of zas}) &  $x^{\chi}$ (section \ref{local basis of coloumb}, after (\ref{inj}))\\ 
 \hline
\end{tabular}
\end{center}

(a) is by Corollary \ref{embed M}.
(b) is by section \ref{def of zastava}.
(c) is not proven yet but is easy to see. It is listed as a heuristic for the proof of theorem \ref{main}. 
(d) is by definition.
(e) is by section \ref{coloumb global basis}.
(f) is by section \ref{global basis for zastava}.

\subsection{Identification of $H_{*}^{T}(\mathcal{R}^{+}_{\leq 1})$ and $\Gamma(C^{\alpha},(q^{\alpha})^{*}\mathcal{I})$}\label{identify projective space}
We have the following decompositions into connected components.
\begin{center}
\begin{tabular}{ |c|c|c| } 
 \hline
 & $\mathcal{R}_{\leq 1}$ & $(q^{\alpha})^{*}Gr(\mathcal{T}^{\alpha})$ \\
 \hline
 connected components  & $\mathcal{R}_{\varpi_{\beta}}$ & $(q^{\alpha})^{*}Gr^{\beta}(\mathcal{T}^{\alpha})$\\ 
  \hline
  \end{tabular}
  \end{center}
The connected components on both sides are indexed by the same data $\beta=(k_i)_{i\in I}$ and $0\leq k_i\leq n_i$. 
On each component, we have
\begin{center}
\begin{tabular}{ |c|c|c| } 
 \hline rank 1 free module
 &  $H_{*}^{T}(\mathcal{R}_{\varpi_{\beta}})$ & $\Gamma(C^{\alpha},(q^{\alpha})^{*}\mathcal{I}^{\alpha,\beta})$ \\
 \hline
 over the ring  & $H^{*}_{T}(\mathcal{R}_{\varpi_{\beta}})$ & $\mathcal{O}((q^{\alpha})^{*}Gr^{\beta}(\mathcal{T}^{\alpha}))$ \\ 
 \hline
 with basis  &  the fundamental class $[\mathcal{R}_{\varpi_{\beta}}]$ & $u^{\beta}\otimes 1$ \\ 
 \hline
\end{tabular}
\end{center}
In the third line, the fundamental class $[\mathcal{R}_{\varpi_{\beta}}]$ is a basis of $H_{*}^{T}(\mathcal{R}_{\varpi_{\beta}})$ over $H^{*}_{T}(\mathcal{R}_{\varpi_{\beta}})$ by Poincare duality (See the paragraph after lemma \ref{bundle over grass}, note there we used $G$-equivariance and here $T$-equivariance).
The basis of $\Gamma(C^{\alpha},(q^{\alpha})^{*}\mathcal{I}^{\alpha,\beta})$ over $\mathcal{O}((q^{\alpha})^{*}Gr^{\beta}(\mathcal{T}^{\alpha}))$ is chosen as $u^{\beta}\otimes 1$ (defined at the beginning of \ref{local basis of zas}\footnote{There, it is considered as a basis element for 
$\Gamma(C^{\alpha},(q^{\alpha})^{*}\mathcal{I}^{\alpha,\beta})$ as an $\mathcal{O}(C^{\alpha})$-module. Here it is a basis as an $\mathcal{O}((q^{\alpha})^{*}Gr^{\beta}(\mathcal{T}^{\alpha,\beta}))$-module.}).

We will identify the rings in lemma \ref{iso of ring} and therefore also the trivialized modules from line 1 of the above table.
\begin{lemma}\label{iso of ring}
There are canonical ring isomorphisms
of the horizontal arrows of the following diagram
\begin{center}
 \begin{tikzcd}
   \mathcal{O}((q^{\alpha})^{*}Gr^{\beta}(\mathcal{T}^{\alpha})) \arrow{r}{\cong} \arrow{d}{res} & H^*_{T}(\mathcal{R}_{\varpi_{\beta}})\cong H^*_{T}(\prod_{i\in 
   I}Gr(n_i,k_i)) \arrow{d}{loc} \\
    \bigoplus_{S,|S|=|\beta|} \mathcal{O}(C^{\alpha}_{reg})\arrow{r}{\cong} &  \bigoplus_{\chi\in W\varpi_{\beta}} \mathcal{H}^*_{T}(t^{\chi})  
\end{tikzcd}
\end{center}
which makes the diagram commute. Here $loc$ is the localization map.

\end{lemma}

\begin{proof}
The isomorphism for the upper horizontal map is clear from theorem \ref{basis for local space} and remark \ref{coh of grass T}, where in the case $I=\{1\}$,
it sends $c^i_l$ to $C_i(S_l)$ and $d^i_k $ to $C_i(Q_k)$.
The definition of the lower horizontal map and the commutativity of the diagram follow from the localization theory of equivariant homology.
\end{proof}
Denote the above isomorphism between these two rank 1 modules (by the above map between rings and trivialization maps from the bases $[\mathcal{R}_{\varpi_{\beta}}]$ and $u^{\beta}\otimes 1$) by $rel$, 
\begin{equation}\label{matching modules}
\Gamma(C^{\alpha},(q^{\alpha})^{*}\mathcal{I}^{\alpha,\beta})\xrightarrow[\cong]{rel} H_{*}^{T}(\mathcal{R}_{\varpi_{\beta}}), \text{ } rel(u^{\beta}\otimes 1)=[\mathcal{R}_{\varpi_{\beta}}].
\end{equation}

\subsection{Matching the local basis}

We recall a lemma in the localization theory of equivariant homology theory.
\begin{lemma}\label{homology module compatible}
Let a torus $T$ act on a complex algebraic variety $X$. The action of $H^{*}_{T}(X)$ on $H_{*}^{T}(X)$ is compatible with localizations, i.e. 

\begin{tikzcd}
H^{*}_{T}(X)\times H_{*}^{T}(X) \arrow[r] \arrow[d, "loc\times loc"'] & H_{*}^{T}(X) \arrow[d, "loc"] \\
H^{*}_{T}(X^T)\times H_{*}^{T}(X^T) \arrow[r]\arrow[d, hookrightarrow] & H_{*}^{T}(X^T)\arrow[d, hookrightarrow] \\
H^{*}_{T}(X^T)\times \mathcal{H}_{*}^{T}(X^T)  \arrow[r] & \mathcal{H}_{*}^{T}(X^T).
\end{tikzcd}

\end{lemma}
Now we state a lemma about the decomposition of the two modules in \ref{matching modules} after localizations.
\begin{lemma}\label{commute for module loc}
The following diagram commutes\\
\begin{tikzcd}
 \Gamma(C^{\alpha},(q^{\alpha})^{*}\mathcal{I}^{\alpha,\beta})\arrow[r,"rel"{above},"\cong"{below}] \arrow[d,"loc"] & H_{*}^{T}(\mathcal{R}_{\varpi_{\beta}})\arrow[d,"loc"]\\
 \bigoplus_{S,|S|=\beta}\Gamma(C_{reg}^{\alpha}, L^{\alpha,S}_{reg})\arrow[r,"rel"{above},"\cong"{below}]& \bigoplus_{\chi\in W\varpi_{\beta}} \mathcal{H}_{*}^{T}(t^{\chi}),
 \\    
\end{tikzcd}\\
where each term is considered as a module over the ring corresponding to the same place in the diagram in lemma \ref{iso of ring}. Moreover, when we decompose the terms on the bottom as $\mathcal{O}(C^{\alpha}_{reg})$(resp. $\mathcal{H}^{T}_{*}(pt)$)-modules, the decomposition is compatible with $rel$, i.e. it maps the $S$-summand to the $\chi(S)$-summand. 
\end{lemma}
\begin{proof}
Follows from lemmas \ref{module loc structrue} and \ref{homology module compatible}.   
\end{proof}

\begin{lemma}\label{match local}
Restricting $rel$ to $C_{reg}^{\alpha}$, 
the basis $s^S$ (section \ref{local basis of zas}) of $\Gamma(C_{reg}^{\alpha}, L^{\alpha,S}_{reg})$ goes to the basis $x^{\chi(S)}$ (section \ref{local basis of coloumb}, after (\ref{inj})) of $\mathcal{H}_{*}^{T}(t^{\chi})$.
\end{lemma}
\begin{proof}
From $$rel(u^{\beta}\otimes 1)=[\mathcal{R}_{\varpi_{\beta}}],$$ decompose 
$u^{\beta}\otimes 1$ and $[\mathcal{R}_{\varpi_{\beta}}]$ into $S$ (or $\chi(S)$ resp.)-component. Since $s^S$ and $x^{\chi(S)}$ are the $S$-component of $u^{\beta}\otimes 1$ and  $[\mathcal{R}_{\varpi_{\beta}}]$ respectively,
by lemma \ref{commute for module loc}, $$rel(s^S)=x^{\chi(S)}.$$
\end{proof}

\subsection{A flatness lemma} \label{lemma}

\begin{lemma}\label{closure lemma}
Let $S$ be a base space and $U\subset S$ an open dense subspace. Let $Y=\mathbb{P}^n_{S}$ over $S$.
Let $X$ over $U$ be a closed subscheme in $Y_U=\mathbb{P}^n_{U}$ and $\overline{X}$ be its closure in $Y.$
Let $\mathcal{X}$ be a closed subscheme flat over $S$.
If $\mathcal{X}|_U=X$, we have $\mathcal{X}=\overline{X}.$    
\end{lemma}

\begin{proof}
Since $\mathcal{X}$ is closed in $Y$, $\overline{X}\subset \mathcal{X}$.
So for $n$ big enough, we have 
$$ \mathcal{O}(n)_{\mathcal{X}/S}\twoheadrightarrow
\mathcal{O}(n)_{\overline{X}/S}.$$
Since $\mathcal{X}$ is flat over $S$, the sheaf $\mathcal{O}(n)_{\mathcal{X}/S}$ is locally free for $n$ big enough.
Now we claim the composition 
$$ \mathcal{O}(n)_{\mathcal{X}/S}\xrightarrow[]{}
\mathcal{O}(n)_{\overline{X}/S}\xrightarrow[]{}\mathcal{O}(n)_{X/S}$$
is injective. Take a section $s$ in the kernel. By the assumption $\mathcal{X}|_{U}=X$, $s|U=0$ but $s$ is a section in the locally free sheaf so it must be $0$.
Hence $$ \mathcal{O}(n)_{\mathcal{X}/S}\xrightarrow[]{}\mathcal{O}(n)_{\overline{X}/S}$$ is also injective so it is a isomorphism for $n>>0$. This implies $\mathcal{X}=\overline{X}.$
\end{proof}    

\subsection{Theorem \ref{main} reduces to comparison on the regular part}
\begin{lemma}\label{local iso}
The dotted arrow exists as an isomorphism, i.e. the isomorphism $rel$ restricts to an isomorphism for the two embeddings.

\begin{tikzcd}
   (q^{\alpha})^{*}Z_{reg}^{\alpha}\arrow[d,dashed]{rel_{reg}}\arrow[r,hook] & (q^{\alpha})^{*}Z^{\alpha}\arrow[r,hook,"i_{\alpha}"] & \mathbb{P}_{C^{\alpha}}((q^{\alpha})^{*}V^{\alpha})\arrow[d, "\text{rel}"', "\simeq"]{rel}\\
    (q^{\alpha})^{*}\mathbf{M}_{reg}^{\alpha} \arrow[r,hook] & (q^{\alpha})^{*}\mathbf{M}^{\alpha} \arrow[r,hook,"j_\alpha"]  & \mathbb{P}_{C^{\alpha}}(H_{*}^{T}(\mathcal{R}^{+}_{\leq 1})^{*}) 
\end{tikzcd}
\end{lemma}

Given this lemma, we can prove theorem \ref{main}.
\begin{proof}[Proof of theorem \ref{main}]
 By \cite[lemma 5.3]{braverman2018towards}, $(q^{\alpha})^{*}\mathbf{M}^{\alpha}$ is flat over $C^{\alpha}.$
 Now we apply lemma \ref{closure lemma}.
 Let $S=C^{\alpha}$ and $U=C^{\alpha}_{reg}.$
 Let $X=rel_{reg}((q^{\alpha})^{*}Z^{\alpha}_{reg})$, $\mathcal{X}=(q^{\alpha})^{*}\mathbf{M}^{\alpha}$   
 and $Y=\mathbb{P}_{C^{\alpha}}(H_{*}^{T}(\mathcal{R}^{+}_{\leq 1})^{*}).$
By lemma \ref{local iso}, we have $\mathcal{X}|_{U}=X$ so applying lemma \ref{closure lemma} gives $\mathcal{X}=X$, i.e., $rel((q^{\alpha})^{*}Z)=(q^{\alpha})^{*}\mathbf{M}^{\alpha} .$
 
\end{proof}
\begin{cor}\label{flat}
The local space $Z_{\kappa(Q)}^{\alpha}$ is flat over $\mathcal{H}^{\alpha}$.
\end{cor}
\begin{proof}
 The space $(q^{\alpha})^{*}\mathbf{M}^{\alpha}$ is flat over $C^{\alpha}$ (see the proof of theorem \ref{main}), so $\mathbf{M}^{\alpha}$ is flat over $\mathcal{H}^{\alpha}$. By theorem \ref{main}, $Z_{\kappa(Q)}^{\alpha}\cong  \mathbf{M}^{\alpha}$ is flat over $\mathcal{H}^{\alpha}$. 
\end{proof}
\subsection{Proof of the comparison lemma on the regular part (lemma \ref{local iso})}
\begin{proof}
Recall that $s^S,S\subset S^{\alpha}$  generate the algebra $A_{0,reg}$ (section \ref{local basis of zas}) and 
$x^{\chi(S)}t,S\subset S^{\alpha}$ generate
the localized Rees algebra $\oplus_m \mathcal{H}^{T_{\mathcal{O}}}((\mathcal{R}_{T,N_T})^{+}_{\leq m})t^m$ (Corollary \ref{generators x chi}). It suffices to prove that these satisfy the same relations when we identify $s^S$ with $x^{\chi(S)}$ by lemma \ref{match local}.

By the multiplication formula (\ref{mult}) in theorem \ref{multi for r^A}, we get the multiplication formula in the Rees algebra, (we need to add $t$ on the right hand side)
$$
r^{\lambda}r^{\mu}=\prod^{n}_{i=1} \xi_{i}^{d(\xi_i(\lambda),\xi_i(\mu))}r^{\lambda+\mu}t.$$
Replace the indices $\lambda,\mu$ by $A,B$ and
denote the coefficient before $r^{\lambda+\mu}$ by $\text{fc}_Q(A,B)$, we rewrite the formula as
$$
r^{A}r^{B}=\text{fc}_Q(A,B)r^{\lambda+\mu}t.$$ 
In particular,
$$
r^{A\cup B}r^{A\cap B}=\text{fc}_Q(A\cup B,A\cap B)r^{\lambda+\mu}t.$$ hence the relations in the localized Rees algebra $\oplus_m \mathcal{H}^{T_{\mathcal{O}}}((\mathcal{R}_{T,N_T})^{+}_{\leq m})t^m$ can be written as
$$r^{A}r^{B}=\frac{\text{fc}_Q(A,B)}{\text{fc}_Q(A\cup B,A\cap B)}
r^{A\cup B}r^{A\cap B}.$$
Recall by lemma \ref{euler factor} that 
$$x^{S}=(Eu_S)^{-1} r^{S}$$
so the relations in terms of $x^{S}$ are
$$x^{A}x^{B}=\frac{\text{fc}_Q(A,B)}{\text{fc}_Q(A\cup B,A\cap B)}\frac{Eu(A\cup B)Eu(A\cap B)}{Eu(A)Eu(B)}
x^{A\cup B}x^{A\cap B}.$$
Now we consider the relations in the algebra $A_{0,reg}.$
Recall in (\ref{local fac}),  
we defined $l(\beta)$ depending on $\kappa.$ 
Now we set $s^{\beta}$ by the same formula as in (\ref{def local  basis}) replacing $l(\beta)$ by 
$$l_{Q}(\beta)=\prod_{a\in E \text{ such that } s(a)=i,t(a)=i',i\neq i' l\in (S_{\beta})_i,j\in (S_{\beta})_{i'}} (x^{i}_l-x^{i'}_j)^{\kappa_{ii'}} \prod_{i\in I,l\neq j,l,j\in(S_{\beta})_i} (x^{i}_l-x^{i}_j)^{\kappa_{ii}},$$
where for an arrow $a\in E$, $s(a)$ is the source and $t(a)$ is the target.
The difference between $l(\beta)$ and $l_{Q}(\beta)$ is possibly a sign which depends on the directions of the arrows in $Q$. 
So accordingly,
in the formula (\ref{induced def of local}) of $s^S$, we replace $l(S)$ by 
$$l_Q(S)=\prod_{a\in E \text{ such that } s(a)=i,t(a)=i',i\neq i', l\in S_i,j\in S_{i'}} (x^{i}_l-x^{i'}_j)^{\kappa_{ii'}} \prod_{i\in I,l\neq j,l,j\in S_i} (x^{i}_l-x^{i}_j)^{\kappa_{ii}}.$$

The relations in the algebra $A_{0,reg}$ for $s^S,S\subset S^{\alpha}$ are 
$$s^A s^B=\frac{l_Q(A)l_Q(B)}{l_Q(A\cup B)l_Q(A\cap B)}s^{A\cup B} s^{A\cap B}.$$ 
Now to check $s^S,S\subset S^{\alpha}$ and 
$x^{\chi(S)}t,S\subset S^{\alpha}$ satisfy the same relations, it suffices to show that 
\begin{equation}\label{equation}
    \frac{\text{fc}_Q(A,B)}{\text{fc}_Q(A\cup B,A\cap B)}\frac{Eu(A\cup B)Eu(A\cap B)}{Eu(A)Eu(B)}=
\frac{l_Q(A)l_Q(B)}{l_Q(A\cup B)l_Q(A\cap B)}.
\end{equation}
This will be an elementary combinatorial calculation.
For a general quiver $Q,$ denote by $Q_s$ the quiver removing all arrows between different vertices and $Q_b$ the removing all arrows.
\begin{itemize}  
\item 
We first prove the case where $Q$ has no edges, i.e. $Q=Q_b.$
In this case $N=0$, so $$\frac{\text{fc}_{Q_b}(A,B)}{\text{fc}_{Q_b}(A\cup B,A\cap B)}=1.$$
 It suffices to show that 
 \begin{equation}\label{combi equation}
 \frac{Eu(A\cup B)Eu(A\cap B)}{Eu(A)Eu(B)}=
\frac{l_{Q_b}(A)l_{Q_b}(B)}{l_{Q_b}(A\cup B)l_{Q_b}(A\cap B)}.
\end{equation}
 From the description of $Eu(S)$ and $l_{Q_b}(S)$, we can assume $I$ is one point. Let $W:=\{1,2,\cdots,n\}$.
Let $A\cap B=C ,A \setminus  B=D,B\setminus A=E, W\setminus (A\cup B) =F$ so $A=C\cup D, B=D\cup E,A\cup B=C\cup D\cup E,A\cap B=D,W=C\cup D\cup E\cup F.$
By lemma \ref{euler factor}(b),
we have 
$$Eu^{-1}(A)=\prod_{l\in A,j\in W\setminus A}(a_l-a_j)=\prod_{(l,j)\in A\times (W\setminus A)}(a_l-a_j).$$
Here, we rewrite
the index set in the second equality as a product $A\times (W\setminus A)$.
It is clear to show \ref{combi equation}, it suffices to compare the index set of the products.
We write  
$$X\overset{\text{supp}}{=}S$$ to mean $$X=\prod_{(l,j)\in S}(a_l-a_j).$$
By this notation,  
$$Eu^{-1}(A)\overset{\text{supp}}{=} A\times (W\setminus A)=(C\cup D)\times (E\cup F)=(C\times E) \cup (C\times F)\cup (D\times E) \cup (D\times F).$$

$$Eu^{-1}(B)=\prod_{l\in B,j\in W\setminus B}(a_l-a_j)\overset{\text{supp}}{=} 
(D\times C)\cup (D\times F)\cup (E\times C)\cup (E\times F).$$
$$Eu^{-1}(A\cup B)=\prod_{l\in A\cup B,j\in W \setminus (A\cup B)}(a_l-a_j)\overset{\text{supp}}{=} 
(C\times F)\cup (D\times F)\cup (E\times F)$$\\
$$Eu^{-1}(A\cap B)=\prod_{l\in A\cap B,j\in  W \setminus (A\cap B)}(a_l-a_j)\overset{\text{supp}}{=} 
(D\times C)\cup (D\times E)\cup (D\times F).$$
So 
$$\frac{Eu^{-1}(A)Eu^{-1}(B)}{Eu^{-1}(A\cup B)Eu^{-1}(a\cap B)}\overset{\text{supp}}{=} (C\times E)\cup(E\times C).$$\\
For $l_{Q_b}$, we have 
$$l_{Q_b}(A)=\prod_{l\in A,j\in A}(a_l-a_j)\overset{\text{supp}}{=} \{(C\times C)\cup(C\times D)\cup(D\times C)\cup(D\times D)\}.$$
$$l_{Q_b}(B)=\prod_{l\in B,j\in B}(a_l-a_j)\overset{\text{supp}}{=} \{(D\times D)\cup(D\times E)\cup(E\times D)\cup(E\times E)\}.$$
$$l_{Q_b}(A\cup B)=\prod_{l\in A\cup B,j\in A\cup B}(a_l-a_j)\overset{\text{supp}}{=} \\$$
$$
\{(C\times C)\cup(C\times  D)\cup(C\times E)\cup(D\times C)\cup(D\times D)\cup(D\times E)\cup(E\times C)\cup(E\times D)\cup(E\times E)\}.$$
$$l_{Q_b}(A\cap B)=\prod_{l\in A\cap B,j\in A\cap B}(a_l-a_j)\overset{\text{supp}}{=} \{(D\times D)\}.$$
So$$
\frac{l_{Q_b}(A)l_{Q_b}(B)}{l_{Q_b}(A\cup B)l_{Q_b}(A\cap B)}\overset{\text{supp}}{=} (C\times E)\cup(E\times C)$$ and therefore (\ref{combi equation}) holds.
\item 
Now we prove the case where $Q=Q_s$ has only self-edges. Again it is clear we can assume $I$ is one point. 
Denote $$l^d_{Q_s}=\frac{l_{Q_s}}{l_{Q_b}}.$$
Now it suffices to show that 
$$\frac{\text{fc}_{Q_s}(A,B)}{\text{fc}_{Q_s}(A\cup B,A\cap B)}=
\frac{l^d_{Q_s}(A)l^d_{Q_s}(B)}{l^d_{Q_s}(A\cup B)l^d_{Q_s}(A\cap B)}.
$$
By remark \ref{mult prop of C factors}, we can reduce to the case where $Q_s$ has one self-loop and it suffices to prove 
$$\frac{\text{fc}_{Q_s}(A,B)}{\text{fc}_{Q_s}(A\cup B,A\cap B)}\overset{\text{supp}}{=}
(C\times E)\cup(E\times C).$$
Now $N=Hom(V, V)$ is the adjoint representation of $GL(V).$ The characters $\xi_i$ is the set of all roots of $GL(V),$ which are all $(a_l-a_j)$ for $l\neq j,l,j \in  W.$
Compute the pairing between $A,B,A\cup B,A\cap B$ and $(a_l-a_j)$ for $l\neq j,l,j \in  W$ and $A,B$ we get the formula.
\item 
Now we prove the general case.
Denote $l^d_{Q}=\frac{l_Q}{l_{Q_s}}$ and $\text{fc}^d_{Q}=\frac{\text{fc}_{Q}}{\text{fc}_{Q_s}}.$
It suffices to show that 
$$\frac{\text{fc}^d_{Q}(A,B)}{\text{fc}^d_{Q}(A\cup B,A\cap B)}=
\frac{l^d_{Q}(A)l^d_{Q}(B)}{l^d_{Q}(A\cup B)l^d_{Q}(A\cap B)}.
$$
Again by remark \ref{mult prop of C factors}, we can reduce to the case where $Q$ has two vertices and only one arrow $1\xrightarrow[]{}2.$ In this case $N=Hom(V_1, V_2)$.
The characters $\xi_i$ consist of $(a^1_l-a^2_j)$ for all $l\leq l \leq n_1,1\leq j \leq n_2.$
Compute the pairing between $A,B,A\cup B,A\cap B$ and $(a^1_l-a^2_j)$ for all $l\leq l \leq n_1,1\leq j \leq n_2$ and $A,B$ we get the formula.
     \end{itemize}
\end{proof}

\section{Appendix: Proof of lemma \ref{ideal in local} (by Ivan Mirković)}
\subsection{Equations of Segre embeddings}
The locality equations in the discrete range are just the Segre embedding equations of a special type $(\mathbb{P}^1)^n\hookrightarrow \mathbb{P}^{2^n-1}$.
The standard list of Segre equations of general type $\prod_{p\in D}\mathbb{P}(V_p)\hookrightarrow \mathbb{P}(V_D)$, is checked based on the case $\mathbb{P}(A)\times \mathbb{P}(B)\hookrightarrow \mathbb{P}(A\otimes B)$ in \ref{segre}. In our case $(\mathbb{P}^1)^n\hookrightarrow \mathbb{P}^{2^n-1}$ the combinatorics is stated in terms of $Gr(D)\subset \mathbb{N}[D]$ in \ref{special segre}.
\subsubsection{Segre embeddings.}\label{segre}
(i) Case of two factors. A choice of coordinates on $A$ and $B$, $x_a,a\in \mathcal{A}$ and $y_b,b\in \mathcal{B},$ gives coordinates $z_{ab}=x_a\otimes y_b$ on $A\otimes B$. A vector $v\in A\otimes B$ can be thought of as an operator $v:A^{*}\xrightarrow[]{} B$ and $v$ is a pure tensor iff the rank of $v$ is $\leq 1.$ In terms of the matrix $(z_{ab})_{\mathcal{A}\times\mathcal{B}}$ this condition is the vanishing of all of its $2\times 2$ minors $z_{ab}z_{a'b'}-z_{ab'}z_{a'b}.$\\
(ii) Segre embedding  map $\prod_{p\in D}\mathbb{P}(V_p)\xrightarrow[]{} \mathbb{P}(V_D).$ Let $x^p_i,i\in \mathcal{B}_{p},$ be the coordinates on vector spaces $V_p,p\in D,$ so that for $X\subset D$ on $V_X\stackrel{\text{def}}{=} \otimes_{p\in X}V_p.$  
We have coordinates $z_{\beta}^{X}\stackrel{\text{def}}{=} \otimes_{p\in D}x_{\beta_p}^p$ indexed by $\beta\in \mathcal{B}_X\stackrel{\text{def}}{=}\prod_{p\in X}\mathcal{B}_p.$ Then the Segre embedding map $\prod_{p\in D}\mathbb{P}(V_p)\hookrightarrow \mathbb{P}(V_D)$ is given by $z_{\beta}^{D}\mapsto \prod_{p\in D}x_{\beta_p}^p.$
We view $\prod_{p\in D}\mathbb{P}(V_p)$ as the intersection of all $\mathbb{P}(V_X)\times \mathbb{P}(V_Y)\subset \mathbb{P}(V_D)$ over all decomposition $D=X\sqcup Y.$\footnote{The claim $\mathbb{P}(A)\times \mathbb{P}(B\otimes C) \cap_{\mathbb{P}(A\times B\times C)} \mathbb{P}(A\times B)\times \mathbb{P}(C)=\mathbb{P}(A)\times \mathbb{P}(B)\times \mathbb{P}(C) $ means that if a vector $v$ in $V=A\otimes B\otimes C$ is a pure tensor for $A\otimes (B\otimes C)$ and $(A\otimes B)\otimes C$, then it is also a pure tensor for $A\otimes B\otimes C.$ Proof. The assumption is $v=a\otimes \alpha=\gamma \otimes c$ with $a\in A, c\in C,\alpha\in B\otimes C,\gamma\in A\otimes B.$ Choose a basis $c_i$ of $C$ with a dual basis $c^i$ and $c=c_0$. For $i\neq 0$ we have $c^i\perp c,$ hence $0=\langle v,c^i\rangle=a\otimes b_i$ hence $b_i=0$. So, $v=a\otimes b_0\otimes c.$ \qed} So, by (i) equations are the minors $\begin{vmatrix}
z_{\alpha'\beta'} & z_{\alpha'\beta''}  \\
z_{\alpha''\beta'}  & z_{\alpha''\beta''} 
\end{vmatrix}$ indexed by data $D=X\cup Y$ and $\alpha',\alpha''\in \mathcal{B}_X, \beta',\beta''\in \mathcal{B}_Y$ (so that $\alpha'\beta'\in \mathcal{B}_D$ etc.).
\begin{rmk}
We can write these equations as $z_\phi z_\psi-z_{\phi^C} z_{\psi^C}$ for $(\phi,\psi,C)\in \mathcal{B}_D^2 \times Gr(D);$ here $C\in Gr(D)$ acts on $\mathcal{B}_{D}^2$ by the unique involution $\sigma_{C}(\phi,\psi)=(\phi^{C},\psi^{C})$ that exchanges the values on $C$, i.e., $\phi^{C}_{p}=\psi_p$ for $p\in C$ and $\psi^{C}_{p}=\psi_p$ for $p\notin C$ (and the same for $\psi^C$).
So, the equation require invariance of products under $Gr(D)$ as a $D$-degeneration of commutativity which is the case $C=D$ as $(\phi^{D},\psi^{D})=(\phi,\psi).$ (However, this $Gr(D)$ action is only defined on a chosen basis.)
\end{rmk}
\subsubsection{The case of $(\mathbb{P}^{1})^{n}\hookrightarrow \mathbb{P}^{2^n-1}.$}\label{special segre} Now we have identifications $\mathcal{B}_p\xleftarrow[\cong]{}\{\emptyset,p\}=Gr(p)$ and therefore $\mathcal{B}_D\xleftarrow[\cong]{} Gr(D)$ (by $\phi \mapsfrom  X$ for $\phi_{p}=\delta_{p\in X}).$  We also embed $Gr(D)$ into a monoid $\mathbb{N}[D].$ (We often denote $z_X$ by $1_X$.)
\begin{cor}
  The Segre equations are now $z_X z_Y=z_U z_V$, indexed by all $X,Y,U,V\subset D$ with $X+Y=U+V \in \mathbb{N}[D].$ 
  \end{cor}
  \begin{proof}
   Involution $\sigma_C, C\subset D$ preserve for each $p\in D$ the multiset $\phi_p,\psi_p$ hence also the sum $\phi_p+\psi_p.$ Conversely, if $X+Y=U+V$ then $(U,V)=\sigma_C(X,Y)$ for $C=\{p\in D;X_p\neq U_p\}.$
     \end{proof}
   \begin{rmk}
(0) One has $z_{i,j}z_{\emptyset}=z_{i}z_{j}$ and inductively $z_{X}z_{\emptyset}^{|X|-1}=\prod_{p\in D} z_p.$
(1) By setting $1_{\emptyset}=1$ one obtains an open affine subspace of $\mathbb{P}(V_D)$ with functions $\mathbb{k}[1_B,\emptyset\neq B\subset D]$. Here, locality equations reduce to $z_X=\prod_{p\in D}z_p$ with solutions $\mathbb{A}^D.$
   \end{rmk}

 \bibliographystyle{alpha}
\bibliography{main.bib}

\end{document}